\documentclass[letterpaper,10pt,reqno,onefignum,onetabnum]{amsart}
\usepackage[english]{babel}
\usepackage{amsmath}
\usepackage{amsthm}
\usepackage{verbatim}
\usepackage{mathrsfs}
\usepackage{bm}
\usepackage[foot]{amsaddr}
\usepackage[dvipsnames]{xcolor}
\usepackage{hyperref}
\hypersetup{
	colorlinks = true,
	linkcolor = OliveGreen,
	anchorcolor = OliveGreen,
	citecolor = OliveGreen,
	filecolor = OliveGreen,
	urlcolor = OliveGreen
}
\usepackage{algorithm}
\usepackage{algorithmic}
\usepackage{float}
\usepackage{lipsum}
\usepackage{amsfonts}
\usepackage{amssymb}
\usepackage{graphicx}
\usepackage{epstopdf}
\usepackage{cases}
\usepackage{multirow}
\ifpdf
\DeclareGraphicsExtensions{.eps,.pdf,.png,.jpg}
\else
\DeclareGraphicsExtensions{.eps}
\fi
\usepackage{verbatim}
\usepackage{mathrsfs}
\usepackage{bm}

\newtheorem{teo}{Theorem}[section]
\newtheorem{prop}[teo]{Proposition}
\newtheorem{lem}[teo]{Lemma}

\newtheorem{pro}[teo]{Problem}
\newtheorem{algo}[teo]{Algorithm}
\newtheorem{asume}[teo]{Assumption}
\newtheorem{notation}[teo]{Notation}
\newtheorem{rem}[teo]{Remark}

\usepackage{lipsum}
\usepackage{amsfonts}
\usepackage{amssymb}
\usepackage{graphicx}
\usepackage{epstopdf}

\usepackage{cases}
\usepackage{multirow}
\usepackage{algorithmic}
\usepackage{tikz}
\usepackage{booktabs}%
\usetikzlibrary{matrix}
\usetikzlibrary{arrows}
\ifpdf
\DeclareGraphicsExtensions{.eps,.pdf,.png,.jpg}
\else
\DeclareGraphicsExtensions{.eps}
\fi
\usepackage{verbatim}
\usepackage{mathrsfs}
\usepackage{bm}
\usepackage{color}
\usepackage{xcolor}
\usepackage{caption}
\usepackage{subcaption}
\usepackage{enumitem}

\newcommand{\N}{\mathbb N}

\newcommand{\R}{\mathbb R}

\renewcommand{\H}{\mathcal{H}}
\newcommand{\G}{\mathcal G}

\newcommand{\id}{\textnormal{Id}}

\newcommand{\weak}{\rightharpoonup}
\newcommand{\ran}{\textnormal{ran}\,}
\newcommand{\dom}{\textnormal{dom}\,}

\newcommand{\zer}{\textnormal{zer}}
\newcommand{\fix}{\textnormal{Fix}\,}
\newcommand{\gra}{\textnormal{gra}\,}

\newcommand{\argm}[1]{\underset{#1}{\argmin\, }}

\newcommand{\scal}[2]{{\left\langle{{#1}\mid{#2}}\right\rangle}}

\newcommand{\menge}[2]{\big\{{#1}~\big |~{#2}\big\}}

\newcommand{\RR}{\ensuremath{\mathbb{R}}}
\newcommand{\RP}{\ensuremath{\left[0,+\infty\right[}}

\newcommand{\RPP}{\ensuremath{\left]0,+\infty\right[}}

\newcommand{\RX}{\ensuremath{\left]-\infty,+\infty\right]}}

\newcommand{\prox}{\ensuremath{\text{\rm prox}\,}}

\newcommand{\weakly}{\ensuremath{\:\rightharpoonup\:}}
\newcommand{\minimize}[2]{\ensuremath{\underset{\substack{{#1}}}%
{\mathrm{minimize}}\;\;#2 }}

\usepackage{geometry}
\numberwithin{equation}{section}

\numberwithin{equation}{section}

\DeclareFontEncoding{FMS}{}{}
\DeclareFontSubstitution{FMS}{futm}{m}{n}
\DeclareFontEncoding{FMX}{}{}
\DeclareFontSubstitution{FMX}{futm}{m}{n}
\DeclareSymbolFont{fouriersymbols}{FMS}{futm}{m}{n}
\DeclareSymbolFont{fourierlargesymbols}{FMX}{futm}{m}{n}
\DeclareMathDelimiter{\nr}{\mathord}{fouriersymbols}{152}{fourierlargesymbols}{147}

\DeclareMathOperator*{\argmin}{arg\,min}
\DeclareMathDelimiter{\nr}{\mathord}{fouriersymbols}{152}{fourierlargesymbols}{147}
\DeclareMathAlphabet{\mathpzc}{OT1}{pzc}{m}{it}

\title[Solution of Mismatched Monotone+Lipschitz Inclusion Problems]{Solution of Mismatched Monotone+Lipschitz Inclusion Problems}

\author{Emilie Chouzenoux$^{\dagger}$}
\author{Jean-Christophe Pesquet$^{\dagger}$}
\author{Fernando Rold\'an$^{\dagger}$}
\address{$^{\dagger}$Universit\'e Paris-Saclay, CentraleSup\'elec, CVN, Inria, Gif-sur-Yvette 91190, France.}
\email{emilie.chouzenoux@inria.fr}
\email{jean-christophe.pesquet@centralesupelec.fr}
\email{fernando.roldan-contreras@centralesupelec.fr}

\begin{document}

\begin{abstract}
In this article, we study the convergence of algorithms for solving 
monotone inclusions  in the presence of adjoint mismatch. The adjoint 
mismatch arises when the adjoint of a linear operator is replaced by 
an approximation, due to computational or physical issues. This 
occurs in inverse problems, particularly in computed tomography. 
In real Hilbert spaces, monotone inclusion problems 
involving a maximally $\rho$-monotone operator, a cocoercive 
operator, and a Lipschitzian operator
can be solved by the {\it Forward-Backward-Half-Forward} and the {\it 
	Forward-Douglas-Rachford-Forward}. We investigate 
the case of a mismatched Lipschitzian operator. We propose variants 
of the two aforementioned methods to cope with the mismatch, and 
establish conditions under which the weak convergence to a solution 
is guaranteed for these variants. The proposed algorithms hence 
enable each iteration to be implemented with a possibly 
iteration-dependent approximation to the mismatch operator, thus 
allowing this operator to be modified at each iteration. Finally, we 
present numerical experiments on a computed tomography example in 
material science, showing the applicability of our theoretical 
findings.
\end{abstract}

\maketitle
 
\section{Introduction}
A rich literature exists on monotone inclusion problems formulated on 
a Hilbert space $\H$
and their deep relations with optimization, game theory, and data 
science
(see 
\cite{bauschkebook2017,Combettes2018MP,CombettesPesquet2021strategies}
and the references therein). In particular, splitting approaches 
have turned out to play a crucial role for solving complex 
formulations combining monotone and linear operators.
A typical monotone inclusion problem involving the sum of several 
operators is the following one:
\begin{pro}\label{prob:problem0}
	Let $A\colon \H \to 2^\H$ be a maximally $\rho$-monotone operator 
	for 
	some $\rho\in \R$, let 
	$C \colon \H \to \H$ be a $\beta$-cocoercive operator for some 
	$\beta \in \RPP$, let $B :\G \to \G$ be a monotone and 
	$\zeta$-Lipschitzian 
	operator for some $\zeta\in \RPP$, let $L: \H \to \G$
	be a linear bounded operator, let $c\in \G$, and let $\alpha \in 
	\RP$. 
	We want
	to
	\begin{equation}\label{eq:problem0}
		\text{find} \quad x \in \H \quad \text{such that} \quad 0 \in 
		Ax+Cx+\alpha L^*(Lx-c)+L^*BLx,  
	\end{equation}
	under the assumption that the set of solutions 
	is nonempty.
\end{pro}
A particular case of this problem is the following optimization one:
\begin{pro}\label{prob:problem0o}
	Let $f\colon \H \to ]-\infty,+\infty[$ be a proper 
	lower-semicontinuous $\rho$-strongly (resp. $(-\rho)$-weakly) 
	convex  function for 
	some $\rho\in \RP$ (resp $\rho < 0)$, let 
	$g\colon \H \to \RR$ be a differentiable convex function
	with a $1/\beta$-Lischitzian gradient for some 
	$\beta \in \RPP$, let $h\colon \H \to \R$ be a differentiable 
	convex function with a $\zeta$-Lipschitzian 
	gradient for some $\zeta\in \RPP$, let $L: \H \to \G$
	be a linear bounded operator, let $c\in \G$, and let $\alpha \in 
	\RP$. 
	Let 
	\begin{equation}
		F\colon x \mapsto f(x)+g(x)+
		\alpha\frac{\|Lx-c\|^2}{2}
		+ h(Lx).
	\end{equation}
	We want
	to
	\begin{equation}\label{eq:problem0o}
		\minimize{x\in \H}{F(x)}
\end{equation}
under the assumption that the set of solutions 
is nonempty.
\end{pro}
Let $\partial_{\rm F} f$ denotes the Fr\'echet subdifferential of $f$.
The equivalence between Problem \ref{prob:problem0o} and Problem \ref{prob:problem0} is obtained
by setting $A = \partial_{\rm F} f$,
$C=\nabla g$, $B = \nabla h$, 
provided that every local minimizer of $F$
is a global minimizer. The latter condition is satisfied when $F$
is convex
(which obviously arises when $\rho \ge 0$).\\
Another example is the following Nash equilibrium problem 
\cite{BricenoCombettes2013} involving $2$ players:
\begin{pro}\label{pr:game}
Let $\H_1$, $\H_2$, $\G_1$, and $\G_2$ be real Hilbert spaces. For 
every $i\in \{1,2\}$, let 
$f_i\colon \H_i \to ]-\infty,+\infty]$
be a proper lower-semicontinuous $\rho_i$-strongly (resp. 
$(-\rho_i)$-weakly) convex function for  some $\rho_i\in \RP$ (resp 
$\rho_i < 0)$ and  let $g_i\colon \H_i \to \RR$ be a differentiable 
convex function
with a $1/\beta_i$-Lischitzian gradient for some 
$\beta_i \in \RPP$. 
Let $R$ be a bounded linear operator
from $\G_2$ to $\G_1$ and, for every
$i\in \{1,2\}$, let $L_i$ be a linear bounded
operators from $\H_i$ to $\G_i$. 
Let $\alpha \in \RPP$ and,
for every $i\in \{1,2\}$, let $Q_i$ be a self-adjoint linear operator 
from $\G_i$ to $\G_i$ 
such that 
$Q_i-\alpha \id_{\G_i}$ is positive.
Let
\begin{align}
	F_1&\colon (x_1,x_2) \mapsto 
	f_1(x_1)+g_1(x_1)
	+\scal{L_1 x_1}{\frac12 Q_1L_1 x_1
		+RL_2x_2}\\
	F_2&\colon (x_1,x_2) \mapsto
	f_2(x_2)+g_2(x_2)
	+\scal{L_2 x_2}{\frac12 Q_2L_2 x_2
		-R^*L_1x_1}.
\end{align}
We want to find $\overline{x}_1\in \H_1$
and $\overline{x}_2 \in \H_2$ such that
\begin{equation}
	\overline{x}_1 = \argm{x_1\in \H_1}
	{F_1(x_1,\overline{x}_2)}, \qquad
	\overline{x}_2 = \argm{x_2\in \H_2}
	{F_2(\overline{x}_1,x_2)},
\end{equation}
under the assumption that such a pair
$(\overline{x}_1,\overline{x}_2)$ exists.
\end{pro}
Assume that, for every solution 
$(\overline{x}_1,\overline{x}_2)$ to
Problem \ref{pr:game}, every local minimizer of $x_1\mapsto 
F_1(x_1,\overline{x}_2)$ (resp. $x_2\mapsto F_2(\overline{x}_1,x_2)$)
is a global minimizer. For example, this condition is satisfied if,
for every $i\in \{1,2\}$,
$x_i \mapsto f_i(x_i)+g_i(x_i)+\frac12 \scal{L_i x_i}{Q_i L_i x_i}$
is convex.
Then, the above game theory problem is an instance of Problem 
\ref{prob:problem0}, where
$\H = \H_1\times \H_2$, $\G= \G_1\times \G_2$,
$\rho = \min\{\rho_1,\rho_2\}$,
$A= A_1\times A_2$,
$(\forall i\in \{1,2\})$ $A_i = \partial_{\rm F} f_i$,
$C\colon (x_1,x_2) \mapsto (\nabla g_1(x_1),\nabla g_2(x_2))$,
$B\colon (y_1,y_2) \mapsto 
(Q_1 y_1+S y_2-\alpha y_1,-S^* y_1+Q_2 y_2-\alpha y_2)$,
$L\colon (x_1,x_2) \mapsto (L_1 x_1,L_2 x_2)$, and $c=0$.\\

In this article, we will be interested in solving the following relaxation of 
Problem~\ref{prob:problem0}:
\begin{pro}\label{prob:problem1}

Let $A: \H \to 2^\H$ be a maximally $\rho$-monotone operator for 
some $\rho\in \R$, let 
$C : \H \to \H$ be a $\beta$-cocoercive operator for some 
$\beta \in \RPP$, let $B :\G \to \G$ be a monotone and 
$\zeta$-Lipschitzian 
operator for some $\zeta\in \RPP$, let $L: \H \to \G$ and $K: \G 
\to\H$ 
be linear bounded operators, let $c\in \G$, and let $\alpha \in \RP$. 
We want
to 
\begin{equation}\label{eq:problem1}
	\text{find} \quad x \in \H \quad \text{such that} \quad 0 \in 
	Ax+Cx+\alpha K(Lx-c)+KBLx,
\end{equation}
under the assumption that the set of solutions 
is nonempty.
\end{pro}
This formulation arises when $L^*$ in Problem \ref{prob:problem0}
is replaced by some approximation $K$, introducing a so-called 
adjoint mismatch.
Such a mismatch is typically encountered in variational approaches 
for solving inverse problems, where $L$ models a degradation process 
and its adjoint often needs to be approximated due to computational 
or physical issues. Adjoint mismatch problems have been 
the topic of a number of recent works where simpler scenarios than 
Problem \ref{prob:problem1} have been considered. 
The importance of adjoint mismatch in computer tomography has been 
early recognized in \cite{ZengGullberg2000}. Then, various methods 
for solving mismatched forms of Problem \ref{prob:problem0o} have 
been investigated in the literature. 
The analysis of the quadratic case  ($f = g = h = 0$) in 
\cite{ElfvingHansen2018,DongHansen2019}  is grounded on algebraic 
tools.
In the context of the randomized Kaczmarz method, affine 
admissibility problems, i.e., $f = g = 0$, $\alpha = 0$, and $h$ is 
the indicator function of a singleton, have been addressed in 
\cite{LorenzRose2018}.
The case when $\rho \ge 0$, $g= \frac{1}{2} \|\cdot\|^2$, and $h = 0$ 
is investigated in \cite{ChouzenouxMismatch2021} by focusing on the 
proximal gradient algorithm. As an extension of 
\cite{ChouzenouxMismatch2021}, a new preconditioning strategy for the 
proximal gradient algorithm is proposed in \cite{Savanier2022}.
The case when $\rho > 0$, $g = 0$, $\alpha = 0$, and the conjugate 
of $h$ is strongly convex is analyzed in \cite{LorenzMMCP2023} by 
using Chambolle-Pock algorithm with fixed and varying step sizes. A 
similar scenario where $\rho \ge 0$, $h = 0$, and $g = \ell \circ M$, 
where $\ell$ is a convex function  and $M$ is a bounded linear 
operator, has been studied in \cite{ChouzenouxMismatch2023} by 
considering the Condat-V\~u \cite{Condat13,Vu13},  Loris-Verhoeven 
\cite{LorisVerhoeven2011}, and Combettes-Pesquet 
\cite{CombettesPesquet2011PD} primal-dual methods. Note that the 
convergence proofs in 
\cite{ChouzenouxMismatch2023,ChouzenouxMismatch2021,Savanier2022}
rely on cocoercivity properties of the underlying operators, while 
this paper puts emphasis on weaker Lipschitz properties.

Since the operator $L^*(\alpha \id_\mathcal{G} + B)L$ is monotone and 
Lipschitzian, the methods proposed in 
\cite{BricenoDavis2018,Malitsky2020SIAMJO,RyuVu2020} can be used to 
solve Problem~\ref{prob:problem0}. In particular, the authors in 
\cite{BricenoDavis2018} proposed a method called {\it 
forward‑backward-half-forward} (FBHF), which  generalizes the {\it 
forward-backward} (FB) splitting
\cite{Goldstein1964,Lions1979SIAM,passty1979JMAA} and the 
{\it forward-backward-forward} (FBF), also called {\it Tseng's 
splitting} \cite{Tseng2000SIAM}. FBHF involves two activations of the 
Lipschitzian operator, one activation of $C$, and one application of 
the resolvent of $A$  (up to some scale factor), at each iteration. 
On the other hand, the \emph{forward-Douglas--Rachford-forward} (FDRF) splitting proposed in 
\cite{RyuVu2020} involves two activations of the Lipschitzian 
operator and one computation of the resolvent of $A$ and of the 
resolvent of $B$, at each iteration. FDRF reduces to the 
Douglas--Rachford splitting \cite{Eckstein1992,Lions1979SIAM} when 
$C=0$ and it reduces to FBF when the Lipschitzian operator is absent.
When dealing with Problem~\ref{prob:problem1},
the monotonicity of the operator $K(\alpha \id_G + B)L$ is not 
guaranteed, and the existing convergence guarantees for the 
previously mentioned methods collapse.

In our work, we revisit FBHF and FDRF
by proposing variants allowing to tackle 
Problem~\ref{prob:problem1} and 
by studying conditions guaranteeing their convergence in this 
context. Additionally, our analysis will be carried out in the case 
when, at each iteration, $K$ itself is not available, but only an 
approximation $K_n$ of it is.
Our results are therefore of potential interest in scenarios where 
$K_n$ corresponds to a learned operator, for example in neural 
network 
architectures based on the unrolling of optimization algorithms 
\cite{Corbineau2020,Prato2021,Savanier2023}.
In our analysis, 
we will also provide evaluations of the error incurred by the adjoint 
mismatch.

The outline of the paper is as follows.
In Section \ref{se:notation}, we briefly introduce the necessary 
notation and mathematical background. Section \ref{se:preliminary} 
provides preliminary results concerning Problem
\ref{prob:problem1}. We also establish two lemmas which will be 
useful to prove convergence results for the considered algorithms. 
Sections \ref{se:FBFH} and
\ref{se:FDRF} are dedicated
to the convergence analysis of splitting methods for solving Problem \ref{prob:problem1} based on FBHF and 
FDRF, respectively. In Section \ref{se:numexp},
we present a numerical comparison of the two algorithms in the resolution of an 
image recovery problem arising in computer tomography.
Some concluding remarks are drawn in 
Section \ref{se:conclu}.

\section{Notation and Background}\label{se:notation}
Throughout this paper $\H$ and $\G$ are real Hilbert spaces with 
scalar 
product $\scal{\cdot}{\cdot}$ and associated norm $\|\cdot 
\|$. The symbols $\weakly$ and $\to$ denote the weak and strong 
convergence, respectively. The identity operator on $\H$ is denoted 
by $\id_\H$. We denote the set of bounded linear operator from $\H$ 
to $\G$ by $\mathcal{B}(\H,\G)$. Given a linear operator $M \in 
\mathcal{B}(\H,\G)$ we denote its adjoint 
by $M^* \in \mathcal{B}(\G,\H)$.
Let $D\subset \H$ be non-empty set and let $T\colon D \rightarrow 
\H$. 
The set of fixed points of $T$ is $\fix T = \menge{x \in D}{x=Tx}$.
Let $\beta \in \left]0,+\infty\right[$. The operator $T$ is 
$\beta-$cocoercive if 
\begin{equation} \label{def:coco}
(\forall x \in D) (\forall y \in D)\quad \langle x-y \mid Tx-Ty 
\rangle 
\geq \beta \|Tx - Ty 
\|^2
\end{equation}
and it is $\beta-$Lipschitzian if 
\begin{equation} \label{def:lips}
(\forall x \in D) (\forall y \in D)\quad \|Tx-Ty\| \leq  \beta\|x 
- y 
\|.
\end{equation}	
When the above inequality holds, the smallest constant $\beta\in 
[0,+\infty[$ allowing it to be satisfied is called the Lipschitz 
constant of $T$ and denoted 
by $\operatorname{Lip}T$.
Let $A\colon\H \rightarrow 2^{\H}$ be a set-valued operator.
The domain, range, zeros, and graph of $A$ 
are 
$\dom\, A = \menge{x \in \H}{Ax \neq  \varnothing}$,
$\ran\, A = \menge{u \in 
\H}{(\exists x \in \H)\,\, u \in Ax}$,  $\zer A = 
\menge{x \in \H}{0 \in Ax}$,
and $\gra A = \menge{(x,u) \in \H \times \H}{u \in Ax}$, 
respectively. Moreover, the inverse of $A$ is given by
$A^{-1} \colon u \mapsto \menge{x \in \H}{u \in Ax}$.
Let $\rho \in \R$, the operator $A$  is $\rho$-monotone if, for every
$(x,u)\in \gra A$ and $(y,v)\in \gra A$ we have
\begin{equation}\label{eq:defrhomon}
\scal{x-y}{u-v} 
\geq \rho \|x-y\|^2.
\end{equation}
Additionally, $A$ is maximally $\rho$-monotone if it is 
$\rho$-monotone and its graph is 
maximal in the sense of 
inclusions among the graphs of $\rho$-monotone operators. In the case 
when $\rho=0$, $A$ is (maximally) monotone, and when $\rho >0$ $A$ is 
strongly (maximally) monotone.
The resolvent of a maximally $\rho$-monotone operator $A$ is defined 
by
$J_A:=(\id+A)^{-1}$ and, if $\rho>-1$, $J_A$ is single valued 
and $(1+\rho)$-cocoercive \cite[Table~1]{BauschkeMoursiXianfu2021}. 
Note that, if $A$ is $\rho$-monotone, then, for every $\gamma \in 
\RPP$, $\gamma A$ is $\gamma\rho$-monotone.

We denote by $\Gamma_0(\H)$ the class of proper lower 
semicontinuous convex functions $f\colon\H\to\RX$. Let 
$f\in\Gamma_0(\H)$.
The Fenchel conjugate of $f$ is 
defined by $f^*\colon u\mapsto \sup_{x\in\H}(\scal{x}{u}-f(x))$ and 
we have
$f^*\in \Gamma_0(\H)$. The Fr\'echet subdifferential of $f$ is the 
maximally monotone operator
$$\partial_F f\colon x\mapsto \menge{u\in\H}{(\forall y\in\H)\:\: 
f(x)+\scal{y-x}{u}\le f(y)},$$
we have that
$(\partial_F f)^{-1}=\partial_F f^*$ 
and that $\zer\,\partial_F f$ is the set of 
minimizers of $f$, which is denoted by $\arg\min_{x\in \H}f$. 

For further properties of monotone operators,
nonexpansive mappings, and convex analysis, the 
reader is referred to \cite{bauschkebook2017}.

\section{Preliminary results}\label{se:preliminary}
By simple calculation, we can show that, for every $M \in 
\mathcal{B}(\G,\H)$,
the operator 
$M(\alpha 
\id_{\G} + B)L$ is 
Lipschitzian. Indeed, for every $(x,y) \in \H^2$,
\begin{equation}
\|M(\alpha 
\id_{\G} + B) L x - M(\alpha 
\id_{\G} + B) L y\| \leq (\alpha \|M\circ L\| + \zeta\|M\|\|L\| ) 
\|x-y\|.
\end{equation}
This applies, in particular, to
$MBL$.
Henceforth, we introduce the following notation.
\begin{notation}
In the context of Problem~\ref{prob:problem1}, 
for every $M \in \mathcal{B}(\G,\H)$, define
$D_M\colon x \mapsto 
\alpha M(Lx-c)+MBLx$, 
$\kappa_M = \operatorname{Lip}(M(\alpha \id_{\G}+B)L)$, and
$\widetilde{\zeta}_{M} = \operatorname{Lip}
(MBL)$.
Let $\lambda_{\min} \in \R$ be defined 
by
\begin{equation}\label{eq:deflmin}
	\lambda_{\min} = \inf \menge{\scal{x}{KLx}}{ x \in \H,\  \|x\| 
		=1}.
\end{equation}
\end{notation}
In order to guarantee the convergence of methods for solving 
Problem~\ref{prob:problem1}, we introduce the following assumptions: 
\begin{asume}\label{assume:1}
In the context of Problem~\ref{prob:problem1}, suppose that
\begin{enumerate}
	\item \label{eq:neqassume1} $D_K\neq 0$,
	\item \label{eq:desassume1} 
	$\hat{\rho}=\rho+\alpha\lambda_{\min}-\widetilde{\zeta}_{L^*-K} 
	\geq 0$,
	\item\label{eq:aproxassume1} $(K_n)_{n\in\N}$ is a sequence of  
	$\mathcal{B}(\G,\H)$ such that, for every $n \in \N$, 
	$\|K_n-K\|\leq \omega_n$,
	where $\{\omega_n\}_{n\in \N}\subset \RP$ and $\sum_{n \in 
		\N} 
	\omega_n < +\infty$.
\end{enumerate} 
\end{asume}	
\begin{rem} In the case when $\alpha > 0$, $B=0$, and $A$ is 
maximally 
monotone ($\rho =0$), 
Assumption~\ref{assume:1} reduces to 
the monotonicity of $KL$, that is $\lambda_{\min}\geq 0$, 
which is a necessary condition for $KL$ to be cocoercive 
\cite[Lemma~3.3]{ChouzenouxMismatch2021}, thus, for ensuring the 
convergence of cocoercive linear mismatch methods proposed 
in  
\cite{ChouzenouxMismatch2023,ChouzenouxMismatch2021}. In 
general, monotone linear operators are not
necessarily cocoercive, for instance, consider the
operator $M \colon \R^2 \to \R^2 \colon (x,y) \mapsto (-y,x)$.
\end{rem}
\begin{prop} \label{prop:MP}
In the context of Problem~\ref{prob:problem1} and  
Assumption~\ref{assume:1}, the following assertions hold:
\begin{enumerate}
	\item \label{prop:MP1} $A+D_K$ is maximally monotone.
	\item \label{prop:MP2} $A+C+D_K$ is maximally monotone.
	\item \label{prop:MP3}  
	Suppose that $\hat{\rho}>0$. Then
	$A+D_K$ is $\hat{\rho}$-strongly monotone and 
	$\zer(A+C+D_K)$ is a singleton.
	
\end{enumerate}
\end{prop}
\begin{proof}\ 
\begin{enumerate}
	\item  In view 
	of the $\rho$-monotonicity of $A$,
	the definition of $\lambda_{\min}$ in \eqref{eq:deflmin}, 
	the 
	Lipschitzianity of $(L^*-K)BL$, the monotonicity of $L^*BL$ 
	(\cite[Proposition~20.10]{bauschkebook2017}), and 
	Assumption~\ref{assume:1}.\ref{eq:desassume1}, we 
	have,
	for 
	every 
	$\big((x,u),(y,v)\big) \in (\gra A)^2$, 
	\begin{align}
		&\scal{x-y}{u+D_Kx-(v+D_Ky)}\nonumber\\
		&= 
		\scal{x-y}{u-v}+\alpha\scal{x-y}{KLx-KLy} \nonumber\\
		&\quad+\scal{x-y}{L^*BLx-L^*BLy}+\scal{x-y}{(L^*-K)(BLx-BLy)}\nonumber\\
		& \geq  \scal{x-y}{u-v}+\alpha\lambda_{\min}\|x-y\|^2 
		-\widetilde{\zeta}_{L^*-K}\|x-y\|^2\nonumber\\
		& \geq 
		\hat{\rho}\|x-y\|^2 \geq 0,
		\label{eq:strongmonotonicity}
	\end{align}
	which shows the monotonicity of $A+D_K$. Now, by \cite[Lemma 
	2.8]{BauschkeMoursiXianfu2021}, $A-\rho \id_{\H}$ is maximally 
	monotone, by 
	Assumption~\ref{assume:1}.\ref{eq:desassume1} and 
	\cite[Example 
	20.34]{bauschkebook2017}, 
	$\hat{\rho}\id_{\H}$ is 
	maximally 
	monotone, by \eqref{eq:deflmin} and 
	\cite[Example 20.34]{bauschkebook2017} 
	$\alpha(KL-\lambda_{\min}\id_{\H}) 
	$ is 
	maximally 
	monotone, by \cite[Lemma 
	2.12]{MoursiGiselsson2021} 
	$(K-L^*)BL+\widetilde{\zeta}_{L^*-K}\id_{\H}$ 
	is 
	$(1/(2\widetilde{\zeta}_{L^*-K}))$-cocoercive 
	with 
	full domain, and, by \cite[Corollary~25.6]{bauschkebook2017} and 
	the full domain of $B$, $L^*BL$ is maximally monotone. Since
	\begin{align}\label{eq:AB2MM} A+D_K = &(A-\rho\id_{\H}) +
		\hat{\rho} \id_{\H}
		+(\alpha KL - 
		\alpha 
		\lambda_{\min}\id_{\H})\nonumber\\
		&\;+ 
		((K-L^*)BL+\widetilde{\zeta}_{L^*-K}\id_{\H})+L^*BL-\alpha Kc,
	\end{align} 
	the maximality of $A+D_K$ follows from 
	\cite[Corollary~25.5]{bauschkebook2017}.
	\item By \ref{prop:MP1}, $A+D_K$ is maximally 
	monotone. Since	$C$ is cocoercive with full domain, the 
	operator $A+D_K+C$ is maximally monotone according to 
	\cite[Corollary 
	25.5]{bauschkebook2017}. 
	\item The strong monotonicity of $A+D_K$ follows directly 
	from \eqref{eq:strongmonotonicity}. In view of 
	the cocoercivity of $C$ and 
	\cite[Corollary~23.37]{bauschkebook2017}, we conclude 
	that $\zer(A+C+D_K)$ is a singleton.
	
\end{enumerate}
\end{proof}
\begin{rem}	   
\item  Proposition~\ref{prop:MP}.\ref{prop:MP1} remains valid 
if  Assumption~\ref{assume:1}.\ref{eq:desassume1} is 
replaced by
\begin{equation}
	\rho+\alpha\lambda_{\min}-\widetilde{\zeta}_K \geq 0.
\end{equation}
Indeed, for every $\big((x,u),(y,v)\big) \in (\gra A)^2$, 
\begin{align*}
	\scal{x-y}{u+D_Kx-(v+D_Ky)} 
	& \geq (\rho+\alpha\lambda_{\min}-\widetilde{\zeta}_K
	)\|x-y\|^2,
\end{align*}
which shows the monotonicity of $A+D_K$. The maximal monotonicity is 
deduced in the same way as in the end of the proof of 
Proposition~\ref{prop:MP}.\ref{prop:MP1}. However, since 
$K$ is 
a surrogate for operator $L^*$, it is expected that 
$\widetilde{\zeta}_K\geq \widetilde{\zeta}_{L^*-K}$. 
\end{rem}
The following proposition provides an estimate of the distance 
between a solution to Problem~\ref{prob:problem0} and a solution to 
Problem~\ref{prob:problem1}.

\begin{prop}\label{prop:solutions}
In the context of Problem~\ref{prob:problem1}, assume that 	
$\rho+\alpha \lambda_{\min}>0$.
Then, there exists a unique solution
$z^*$ to Problem \ref{prob:problem0}.
Furthermore, every solution $z$ to Problem~\ref{prob:problem1} is 
such that
\begin{equation}
	\|z-z^*\| \leq \frac{1}{\rho+\alpha \lambda_{\min}}  
	\|L^*-K\|\,
	\|\alpha (L z-c) +BLz\|.
\end{equation} 
\end{prop}

\begin{proof}
Since $\rho+\alpha \lambda_{\min}>0$, it follows from 
Proposition~\ref{prop:MP}.\ref{prop:MP3}
when $K=L^*$ and the cocoercivity of $C$ that $A+C+D_{L^*}$ is 
$(\rho+\alpha \lambda_{\min})$-strongly monotone and 
$\zer(A+C+D_{L^*})$ is a singleton $\{z^*\}$. Let $z \in 
\zer(A+C+D_K)$.
Then
\begin{equation}
	z=J_{A+C+D_{L^*}}(z+D_{L^*-K}z)\quad  \textnormal{ and } \quad z^*=J_{A+C+D_{L^*}}(z^*).
\end{equation}
Since $J_{A+C+D_{L^*}}$ is 
Lipschitzian with constant $1/(1+\rho+\alpha \lambda_{\min})$
\cite[Proposition 23.13]{bauschkebook2017}, we deduce that
\begin{align}
	\|z-z^*\| &\leq 
	\frac{1}{1+\rho+\alpha 
		\lambda_{\min}}\|z+\alpha(L^*-K)(Lz-c)+(L^*-K)BLz-z^*\|\nonumber\\
	&\leq \frac{1}{1+\rho+\alpha 
		\lambda_{\min}}\big(\|z-z^*\|+\|L^*-K\|\, \|\alpha (Lz-c)+BLz\|).
\end{align}
The result follows from the last inequality.	
	%
	%
	%
\end{proof}	
The following lemmas will play a prominent role to prove convergence 
properties
of our proposed methods for solving 
Problem~\ref{prob:problem1}.
\begin{lem}\label{lemma:QQn}
Let $I\subset \RPP$ and let $S$ be a nonempty subset of $\H$. Suppose 
that, for every $\gamma \in I$, 
$Q^{\gamma}:\H \to \H$ is such that there exists a function 
$\phi^{\gamma}\colon \H \to 
\RP$ satisfying, for every $z \in \H$ 
and $z^* \in S$,
\begin{equation}\label{eq:lemQQn1}
	\|Q^{\gamma}z-z^*\|^2\leq \|z-z^*\|^2-\phi^{\gamma}(z).
\end{equation}
For every $z^* \in S$,
let $\{\varpi_n(z^*)\}_{n\in \N}\subset \RP$ and 
$\{\eta_n(z^*)\}_{n\in \N}\subset \RP$
be such that $\sum_{n \in 
	\N} 
\varpi_n(z^*)  < +\infty$
and $\sum_{n \in 
	\N} 
\eta_n(z^*) < +\infty$. 
For 
every $n \in \N$ and $\gamma \in I$, let $Q_n^{\gamma}:\H \to 
\H$ be such 
that
\begin{equation}\label{eq:lem2}
	(\forall z \in \H)\quad 
	\| Q_n^{\gamma} z - Q^{\gamma}
	z\| \leq  
	\varpi_n(z^*)  \| z 
	-z^*  \|
	+ \eta_n(z^*) .
\end{equation} 
Let $\{\gamma_n\}_{n \in \N} \subset I$, let 
$z_0 
\in 
\H$, and define the sequence $(z_n)_{n 
	\in \N}$ recursively by 
\begin{equation}
	\label{eq:algolemma}
	(\forall n \in \N) \quad
	\begin{aligned}
		&z_{n+1} =  Q_n^{\gamma_n} z_n.
	\end{aligned}
\end{equation} 
Then, the following assertions 
hold:
\begin{enumerate}
	\item\label{lemma:QQn1} $(\|z_n-z^*\|)_{n \in \N}$ is 
	convergent.
	\item\label{lemma:QQn2}  $\sum_{n \in \N} 
	\|z_{n+1}-Q^{\gamma_n}z_n\| < 
	+\infty$.
	\item\label{lemma:QQn3}  
	$\sum_{n\in \N} \phi^{\gamma_n} (z_n) <+\infty$.  
	\item\label{lemma:QQn4}  Suppose that every weak sequential 
	cluster point of 
	$(z_n)_{n 
		\in \N}$ belongs to
	$S$. Then $(z_n)_{n 
		\in \N}$ converges weakly to a point in $S$. 
\end{enumerate}
\end{lem}

\begin{proof}
Let $z^* \in S$. 
\begin{enumerate}
	\item By \eqref{eq:lemQQn1} 
	applied to $z=z_n$ and $\gamma=\gamma_n$ we obtain
	\begin{align}\label{eq:lemp1}
		\|Q^{\gamma_n}z_n- z^*\|^2 \leq & 
		\|z_{n}- 
		z^*\|^2 -\phi^{\gamma_n}(z_n).
	\end{align}
	In particular, 
	\begin{align}
		\|Q^{\gamma_n}z_n- z^*\| \leq & 
		\|z_{n}- 
		z^*\|.
	\end{align}
	Additionally, it follows from \eqref{eq:lem2} that 
	\begin{equation}\label{eq:lemp2}
		\| Q_n^{\gamma_n} z_n - Q^{\gamma_n}z_n
		\| \leq  \varpi_n(z^*) \| z_n
		-z^*  \|
		+ \eta_n(z^*) .
	\end{equation}
	Thus, 
	\begin{align}
		\|z_{n+1}-z^*\| &\leq  \|Q^{\gamma_n} z_{n}-z^*\| + \| 
		Q_n^{\gamma_n} z_n - Q^{\gamma_n}z_n\| 
		\label{eq:lemp30}\\
		&\leq  (1+\varpi_n(z^*) )\| z_n 
		-z^*  \|
		+\eta_n(z^*).\label{eq:lemp3}
	\end{align} 
	Therefore, from \cite[Lemma 5.31]{bauschkebook2017}, we conclude 
	that 
	$(\|z_{n}-z^*\|)_{n \in \N}$ is convergent.
	\item We deduce from \ref{lemma:QQn1} that
	$\delta=\sup_{n \in 
		\N} 
	\|z_n-z^*\| < +\infty$. Since 
	$(\varpi_n(z^*))_{n\in \N}$ and $(\eta_n(z^*) )_{n\in \N}$
	are summable sequences, we
	conclude  from \eqref{eq:lemp2}
	that $\sum_{n \in \N} \|Q^{\gamma_n}_n 
	z_n-Q^{\gamma_n}z_n\| < 
	+\infty$.
	\item By using Cauchy-Schwarz inequality, it follows from  
	\eqref{eq:lemp1} that
	\begin{align}\label{eq:lempdesfej}
		\|z_{n+1}-z^*\|^2
		& =  \|Q^{\gamma_n} z_{n}-z^*\|^2 + 
		2\scal{Q^{\gamma_n} z_{n}-z^*}{Q^{\gamma_n}_n 
			z_n - Q^{\gamma_n} 
			z_n} 	\nonumber\\
		& \hspace{5cm} + \| Q^{\gamma_n}_n z_n - Q^{\gamma_n} z_n\|^2 \nonumber\\
		& \leq 	\|z_{n}- 
		z^*\|^2 -\phi^{\gamma_n}(z_n)+ 
		2\|Q^{\gamma_n} z_{n}-z^*\| \|Q^{\gamma_n}_n 
		z_n - Q^{\gamma_n} 
		z_n\| \nonumber\\
		&\hspace{5cm}+ \| Q^{\gamma_n}_n z_n - Q^{\gamma_n} 
		z_n\|^2\nonumber\\
		& \leq 	\|z_{n}- 
		z^*\|^2 -\phi^{\gamma_n}(z_n)+ 
		2\delta \|Q^{\gamma_n}_n 
		z_n - Q^{\gamma_n} 
		z_n\|\\
		& \hspace{5cm}  + \| Q^{\gamma_n}_n z_n - Q^{\gamma_n} z_n\|^2. \nonumber
	\end{align}
	In addition, according to \ref{lemma:QQn2},
	\begin{equation}
		\sum_{n\in \N} 2\delta \|Q^{\gamma_n}_n 
		z_n - Q^{\gamma_n} 
		z_n\| + \| Q^{\gamma_n}_n z_n - Q^{\gamma_n} z_n\|^2 < 
		+\infty.
	\end{equation}
	Then, by invoking again \cite[Lemma 5.31]{bauschkebook2017}, we 
	conclude that
	$\sum_{n\in \N} \phi^{\gamma_n}(z_n) < +\infty$.
	\item  Eq. \eqref{eq:lempdesfej} shows that 
	$(z_n)_{n\in \N}$ is a quasi-Fej\`er sequence with respect to $S$.
	The weak convergence of $(z_n)_{n \in \N}$ thus follows 
	\cite[Theorem~5.33(iv)]{bauschkebook2017}.
\end{enumerate}  
\end{proof}

\begin{lem}\label{lemma:seq}
Let $(\vartheta,\overline{\eta})\in ]0,1[^2$, let $\eta_0 \in \RP$, 
let $\{\varpi_n\}_{n\in \N}\subset \RP$  be such that 
$\lim_{n \to +\infty} \varpi_n =0$, and let $\{a_n\}_{n\in \N}\subset 
\RP$ be such that
\begin{equation}\label{eq:lemmaseq}
	a_{n+1} \leq (\vartheta+\varpi_n)a_n+\eta_0\overline{\eta}^n.
\end{equation}
Then, $(a_n)_{n \in \N}$ converges linearly to $0$.
\end{lem}

\begin{proof}
Since $(\varpi_n)_{n\in \N}\subset \RP$ 
converges to zero and $\vartheta < 1$, there exist $n_0 \in \N$ and 
$\overline{\vartheta}\in ]\vartheta,1[$
such that, for every $n\ge n_0$,
\begin{equation}
	a_{n+1}  
	\leq \overline{\vartheta} a_n 
	+\eta_0\overline{\eta}^n.
\end{equation}
We deduce that, for every $n> n_0$,
\begin{equation}
	a_{n} \leq \overline{\vartheta}^{n-n_0} a_{n_0} 
	+\eta_0\sum_{m=n_0}^{n-1}\overline{\eta}^m\overline{\vartheta}^{n-m-1}.
\end{equation}
Without loss of generality, it can be assumed that 
$\overline{\vartheta}\neq \overline{\eta}$.
We have then, for every $n> n_0$,
\begin{align}
	a_{n} &\le
	\overline{\vartheta}^{n-n_0} a_{n_0}
	+\eta_0\overline{\eta}^{n_0}
	\frac{\overline{\vartheta}^{n-n_0}-\overline{\eta}^{n-n_0}}
	{\overline{\vartheta}-\overline{\eta}}\nonumber\\
	& \le  \left(a_{n_0}+2\frac{\eta_0\overline{\eta}^{n_0}}{
		|\overline{\vartheta}-\overline{\eta}|}
	\right)
	\max\{\overline{\vartheta},\overline{\eta}\}^{n-n_0},
\end{align}
which shows the linear convergence of
$(a_n)_{n\in \N}$ to $0$.
\end{proof}

\begin{lem}\label{lemma:QQnlin}
Let $I\subset \RPP$, let $z^*\in \H$,
and let $(\vartheta,\overline{\eta})\in ]0,1[^2$.
Suppose that, for every $\gamma \in I$, 
$Q^{\gamma}\colon \H \to \H$ is such that, for every $z \in \H$,
\begin{equation}\label{eq:lem1}
	\|Q^{\gamma}z-z^*\|^2\leq \vartheta\|z-z^*\|^2.
\end{equation}
Let $\{\varpi_n(z^*)\}_{n\in \N}\subset \RP$  be such that 
$\lim_{n \to +\infty} 
\varpi_n(z^*)  =0$
and let $\eta_0(z^*) \in \RP$.
For every $n \in \N$ and $\gamma \in I$, let $Q_n^{\gamma}:\H \to 
\H$ be such 
that \eqref{eq:lem2}
holds where $\eta_n(z^*) 
= \eta_0(z^*) \overline{\eta}^n$.
Let $\{\gamma_n\}_{n \in \N} \subset I$ and let 
$z_0 
\in 
\H$. Then the sequence $(z_n)_{n \in \N}$ defined by  
\eqref{eq:algolemma} converges linearly to $z^*$.
\end{lem}
\begin{proof}
It follows from \eqref{eq:lemp30} that,
for every $n\in\N$,
\begin{equation}
	\|z_{n+1}-z^*\|   
	\leq  (\vartheta+\varpi_n(z^*) )\| z_n 
	-z^*  \|
	+\eta_n(z^*).
\end{equation}
The result then follows from Lemma~\ref{lemma:seq}.
\end{proof}

\section{Forward-Backward-Half Forward Splitting}
\label{se:FBFH}
In this section, we will consider the following variant of the FBHF 
algorithm.
\begin{algo}\label{algo:BAD}
In the context of Problem~\ref{prob:problem1}, let $\{\gamma_n\}_{n 
	\in \N} \subset
\RPP$ be such that $(\forall n\in \N)$
$\gamma_n \rho > -1$, and let $z_0 \in \H$. Consider
the iteration
\begin{equation}
	\label{e:algon}
	(\forall n \in \N) \quad \left\lfloor
	\begin{aligned}
		&u_n =  D_{K_n} z_n\\
		&y_n = z_n-\gamma_n
		(C z_n + u_n)\\
		&x_n = J_{\gamma_n A} (y_n)\\
		&z_{n+1} = x_n + \gamma_n (u_n - 
		D_{K_n} x_n).
	\end{aligned}
	\right.
\end{equation} 
\end{algo}

\begin{notation}
In the context of Problem~\ref{prob:problem1}, for every $\gamma \in 
\RPP$ such that $\gamma \rho > -1$, define the 
operators
\begin{equation}\label{eq:defOp}
	S^\gamma=J_{\gamma 
		A} (\id_{\H}- 
	\gamma (C+D_K)), \quad T^\gamma= 
	(\id_{\H}-\gamma D_K)\circ S^\gamma  + \gamma D_K
\end{equation}
and, for every $n \in \N$,
\begin{equation}\label{eq:defOpn}
	S^{\gamma}_n=J_{\gamma 
		A} (\id_{\H}- 
	\gamma (C+D_{K_n})), \quad T^{\gamma}_n= 
	(\id_{\H}-\gamma D_{K_n})\circ S^{\gamma}_n  + \gamma D_{K_n}.
\end{equation}
Additionally, let $\chi \in \left] 0,
\min\left\{2{\beta},{1}/\kappa_K\right\}\right[$ be defined by
\begin{equation}\label{e:chipre}
	\chi = 
	\begin{cases}
		\displaystyle \frac{4{\beta}}{1+\sqrt{1+16{\beta}^2
				\kappa_K^2}}
		& \mbox{if $\rho \geq 0$}\\
		\displaystyle  \min\left\{\frac{4{\beta}}{1+\sqrt{1+16{\beta}^2
				\kappa_K^2}},-\frac{1}{\rho}\right\} & \mbox{otherwise.}
	\end{cases}
\end{equation}
\end{notation}
\begin{prop} \label{prop:BD}
	In the context of Problem~\ref{prob:problem1} and  
	Assumption~\ref{assume:1}, 
	let $\gamma \in [\varepsilon,\chi-\varepsilon]$, for some 
	$\varepsilon \in
	\left]0,\chi/2\right[$. Then, the following assertions hold:
	\begin{enumerate}
		\item \label{prop:BD1} $ \zer (A+C+D_K) 
		= \fix 
		T^{\gamma}$
		\item \label{prop:BD3}  For every $z \in \H$ and every $z^* 
		\in \fix T^{\gamma}$		
		\begin{align}\label{eq:propBD3}
			\quad \|T^{\gamma} z- z^*\|^2 \leq & \|z- 
			z^*\|^2 -\kappa_K^2\varepsilon^2
			\|z-S^{\gamma} 
			z\|^2-\frac{2\beta\varepsilon^2}{\chi}\|Cz-Cz^*\|^2.
		\end{align}
		\item \label{prop:BD4}  Suppose that 
		$\hat{\rho}=\rho+\alpha\lambda_{\min}-\widetilde{\zeta}_{L^*-K} 
		>0$. 
		Then $\fix 
		T^{\gamma}$ is a singleton $\{z^*\}$ and, for every  $z \in 
		\H$,	
		\begin{equation}\label{eq:propBD4}
			\quad \|T^{\gamma} z- z^*\| 
			\le \sqrt{1-\varepsilon\min\{\kappa_K^2\varepsilon/2, 
				\hat{\rho}
				\}}\,\|z- 
			z^*\|.
		\end{equation}	
	\end{enumerate}
\end{prop}
\begin{proof}\ 
	\begin{enumerate}
		\item The property directly follows  from the Lipschitzian 
		property of 
		$D_K$ 
		and 
		\cite[Proposition~2.1.1]{BricenoDavis2018}.
		\item \label{prop:BD3proof}	Note that,  if $z^* \in 
		\zer(A+C+D_K)$, then $-\gamma C 
		z^* \in 
		\gamma (A+ D_K) z^*$. Additionally, by defining 
		$y=z-\gamma(C+D_K)z$ 
		and $x=S^\gamma z=J_{\gamma A} y$, we have  $y-x+\gamma D_K x \in 
		\gamma 
		(A +D_K) 
		x$. Therefore, the monotonicity of $A+D_K$ established in 
		Proposition~\ref{prop:MP}.\ref{prop:MP1} yields 
		\begin{equation}\label{eq:monoAD}
			0 \leq \scal{x-z^*}{y-x+\gamma 
				D_K 
				x+\gamma C z^*} 
		\end{equation}
		and we deduce that
		\begin{align}
			&\scal{x-z^*}{x-y-\gamma D_K x}\nonumber\\
			&=  
			\scal{x-z^*}{\gamma Cz^*}
			-\scal{x-z^*}{y-x+\gamma 
				D_K 
				x+\gamma C z^*}\nonumber\\
			& \leq \scal{x-z^*}{\gamma Cz^*}.
		\end{align}
		By proceeding similarly to the proof of \cite[Proposition 
		2.1.3]{BricenoDavis2018},
		\begin{align}
			\|T^{\gamma} z- z^*\|^2&=\|x-z^*+\gamma (D_K z - D_K x)\|^2\nonumber\\
			&\le \|x-z^*\|^2 + 2\gamma \scal{x-z^*}{Cz^*}
			+ 2\scal{x-z^*}{z-x-\gamma Cz}\nonumber\\
			&\hspace{6cm}
			+ \gamma^2 \| D_K z - D_K x \|^2\nonumber\\
			& = \|z-z^*\|^2-\|z-x\|^2+
			2\gamma \scal{x-z^*}{Cz^*-Cz}
			\nonumber\\
			&\hspace{6cm}
			+ \gamma^2 \| D_K z - D_K x \|^2
			\label{e:BDmaj1} 
		\end{align}
		By using the cocoercivity of $C$, for every
		$\eta \in \RPP$,
		\begin{align}
			2\gamma\scal{x-z^*}{Cz^*-Cz}&\le 
			2\gamma\scal{x-z}{Cz^*-Cz}-2\gamma \beta 
			\|Cz^*-Cz\|^2\nonumber\\
			& \le \eta\|x-z\|^2+
			\gamma \left(\frac{\gamma}{\eta}-2 \beta\right)
			\|Cz^*-Cz\|^2.
			\label{e:BDmaj2} 
		\end{align}
		Combining \eqref{e:BDmaj1}, \eqref{e:BDmaj2}, and using the fact 
		that $D_K$ is $\kappa_K$-Lipschitz leads to
		\begin{align}
			&\|T^{\gamma} z- z^*\|^2\nonumber\\ 
			&\le \|z-z^*\|^2-(1-\eta-\gamma^2 \kappa_K^2)\|z-x\|^2-
			\gamma \left(2\beta-\frac{\gamma}{\eta}\right)\|Cz^*-Cz\|^2.
			\label{e:BDmaj2next} 
		\end{align}
		Let us  choose $\eta<1$ such that 
		$\chi_{0} = \frac{\sqrt{1-\eta}}{\kappa_K} = 2 \beta \eta$ where
		$\chi_{0}=4\beta/(1+\sqrt{1+16{\beta}^2 \kappa_K^2})$.
		It follows from \eqref{e:BDmaj2next} that
		\begin{align}
			&\|T^{\gamma} z- z^*\|^2\nonumber\\ 
			&\le \|z-z^*\|^2-\kappa_K^2 (\chi_{0}^2-\gamma^2 )\|z-x\|^2-
			2\beta\gamma 
			\left(1-\frac{\gamma}{\chi_{0}}\right)\|Cz^*-Cz\|^2.
		\end{align}
		By observing that $\chi \le \chi_{0}$
		and taking into account the domain of variations of $\gamma$,
		\eqref{eq:propBD3} is deduced.
		\item From \ref{prop:BD1} and 
		Proposition~\ref{prop:MP}.\ref{prop:MP2}, 
		we conclude that $\fix 
		T^{\gamma}$ is a singleton. 
		The strong monotonicity of $A+D_K$ allows us to
		obtain the following inequality:
		\begin{equation}\label{eq:smonoAD}
			\gamma \hat{\rho}\|x-z^*\|^2 \leq 
			\scal{x-z^*}{y-x+\gamma 
				D_K 
				x+\gamma C z^*}.
		\end{equation}
		Hence, by proceeding similarly to the proof of \ref{prop:BD3},
		we 	obtain
		\begin{multline}
			\|T^{\gamma} z- z^*\|^2 \\ \leq \|z- 
			z^*\|^2 -\kappa_K^2\varepsilon^2
			\|z-S^{\gamma} 
			z\|^2-\frac{2\beta\varepsilon^2}{\chi}\|Cz-Cz^*\|^2
			-2\hat{\rho}\gamma\|x-z^*\|^2.
		\end{multline}
		Therefore, since $\gamma \geq \varepsilon$,
		\begin{align*}
			\quad \|T^{\gamma} z- z^*\|^2 \leq & \|z- 
			z^*\|^2 -\kappa_K^2\varepsilon^2
			\|z-S^{\gamma} 
			z\|^2-2\hat{\rho}\varepsilon\|S^{\gamma} 
			z-z^*\|^2\\
			& \leq \|z- 
			z^*\|^2 -\min\{\kappa_K^2\varepsilon^2, 
			2\hat{\rho}\varepsilon\} (
			\|z-S^{\gamma} 
			z\|^2+\|S^{\gamma} 
			z-z^*\|^2)\\
			& \leq \|z- 
			z^*\|^2 
			-\frac{\varepsilon}{2}\min\{\kappa_K^2\varepsilon,2\hat{\rho}\}
			(
			\|z-z^*\|^2)\\
			& = \left(1-\varepsilon\min\{\kappa_K^2\varepsilon/2, 
			\hat{\rho}
			\}\right)\|z- 
			z^*\|^2.
		\end{align*}
	\end{enumerate}
\end{proof}
\begin{prop}\label{prop:desop}
	Consider the operators 
	defined in 
	\eqref{eq:defOp} and \eqref{eq:defOpn}.
	Then, there exists
	$(\theta_1,\theta_2,\theta_3,\theta'_3,\theta_4,\theta'_4,\theta''_4)\in
	\RPP^7$ such that,
	for every
	$(z,z^*) \in  \H^2$, for every $\gamma \in ]0,\chi[$,
	and for every $n\in \N$,
	the following inequalities hold:
	\begin{enumerate}
		\item \label{prop:desop1} 
		$\|D_{K_n} z - D_Kz\|\leq \omega_n(\theta_1\|z-z^*\|
		+\|(\alpha \id_{\G} + B)Lz^*\|)$ 
		
		\item \label{prop:desop2}
		$\|S^{\gamma}_n z-S^{\gamma}z\| \leq \frac{1}{1+ 
			\gamma\rho}\|D_{K_n} z- D_Kz\|$
		
		\item \label{prop:desop3} 
			$\|S^{\gamma}_n z 
			-S^{\gamma} z^*  \|\leq \frac{1}{1+ \gamma\rho}(\theta_2 \|z 
			- z^*  \| +  \|D_{K_n} z- D_Kz\|)$
			\item \label{prop:desop4}
				%
				$
				\| D_{K_n} 
				S^{\gamma}_n z - 
				D_K S^{\gamma}z\| \leq 
				\omega_n\left(\frac{\theta_3}{1+\gamma \rho}\| z 
				-z^*  \|
				+\|(\alpha \id_{\G} + B)LS^{\gamma} z^* \|\right.$\\
				\hspace*{\fill}$
				\left.+\frac{\theta'_3}{1+\rho \gamma}\|(\alpha \id_{\G} + 
				B)Lz^*\|\right)$

				%
			\item \label{prop:desop5}
				$
				\| T_n^{\gamma} z - T^{\gamma} 
				z\| \leq  \frac{\omega_n}{1+\gamma \rho}(\theta_4\| z 
				-z^*  \|
				+\theta'_4\|(\alpha \id_{\G} + B)LS^{\gamma} z^* \| 
				+\theta''_4\|(\alpha \id_{\G} + B)Lz^*\|)$.
			\end{enumerate}
		\end{prop}
		\begin{proof}
			First note that, in view of Assumption~\ref{assume:1}, 
			$\overline{\omega}=\sup_{n\in \N} \omega_n <+\infty$. Let 
			$(z,z^*)\in \H^2$ and let $n\in \N$.
			\begin{enumerate}
				\item It follows from 
				Assumption~\ref{assume:1}.\ref{eq:aproxassume1} that
				\begin{align*}
					\|D_{K_n} z - D_K z\|	&= \|K_n(\alpha \id_{\G} + 
					B)Lz-K(\alpha 
					\id_{\G} + 
					B)Lz\|\\
					&\leq \|K_n-K\|\|(\alpha \id_{\G} + B)Lz\|\\
					&\leq \omega_n\|(\alpha \id_{\G} + B)Lz\|\\
					&\leq \omega_n(\|(\alpha \id_{\G} + B)Lz-(\alpha 
					\id_{\G} + 
					B)Lz^*\|\\
					&\hspace{4cm}+\|(\alpha \id_{\G} + B)Lz^*\|)\\
					&\leq \omega_n((\alpha+\zeta)\|L\|\|z-z^*\|
					+\|(\alpha \id_{\G} + B)Lz^*\|).
				\end{align*}
				The result follows by setting 
				\begin{equation}\label{e:deftheta1}
					\theta_1=(\alpha + \zeta)\|L\|.
				\end{equation}
				\item It follows the $(1+\gamma 
				\rho)^{-1}$-Lipschitzianity of 
				$J_{\gamma A}$  that 	
				\begin{align*}
					\|S^{\gamma}_n z-S^{\gamma}z\| &= \|J_{\gamma A} 
					(\id_{\H}- 
					\gamma 
					(C+D_{K_n}))z- J_{\gamma A} (\id_{\H}- \gamma 
					(C+D_K))z\|\\
					&\leq \frac{1}{1+\rho \gamma}\| (\id_{\H}- \gamma 
					(C+D_{K_n}))z- (\id_{\H}- \gamma 
					(C+D_K))z\|\\
					& =\frac{1}{1+\rho \gamma}\|D_{K_n} z- D_Kz\|.
				\end{align*}
				\item Similarly, it follows from \ref{prop:desop2} and 
				the 
				Lipschitzianity  of $J_{\gamma A}$ that 
				\begin{align*}
					\|S^{\gamma}_n z 
					-S^{\gamma} z^*  \| & \leq	\|S^{\gamma} z
					-S^{\gamma} z^*  \|+ \|S^{\gamma}_n z 
					-S^{\gamma} z  \|\\
					& \leq	\frac{1}{1+\gamma 
						\rho}(\|(\id_{\H}-\gamma(C+D_K)) z 
					-(\id_{\H}-\gamma(C+D_K)) z^*  \| \\
					&\hspace{6.5cm}+ \|D_{K_n} z- 
					D_Kz\|)\\
					& \leq	\frac{1}{1+\gamma 
						\rho}\big((1+\gamma(\beta^{-1}+\kappa_K)) \| z 
					- z^*  \|+ \|D_{K_n} z- D_K z\|\big).
				\end{align*}
				The conclusion follows by defining 
				$\theta_2=1+\chi(\beta^{-1}+\kappa_K)$.
				\item	It follows from \ref{prop:desop1}, the 
				Lipschitzian property 
				of $D_K$, \ref{prop:desop3}, and 
				\ref{prop:desop2} that
				\begin{align*}
					&\| D_{K_n} 
					S^{\gamma}_n z - 
					D_K S^{\gamma}z\|\\ 
					&= \|D_{K_n} 
					S^{\gamma}_n z - D_K 
					S^{\gamma}_n z + D_K
					S^{\gamma}_n z-
					D_KS^{\gamma}z \|\\
					& \leq   \|D_{K_n} 
					S^{\gamma}_n z - D_K 
					S^{\gamma}_n z \|+ \|D_K
					S^{\gamma}_n z-
					D_K S^{\gamma}z \|\\
					& \leq   \omega_n (\theta_1\|S^{\gamma}_n z 
					-S^{\gamma} z^*  \|
					+\|(\alpha\id_{\G} + B)LS^{\gamma} z^* \|)
					+ \kappa_K\|
					S^{\gamma}_n z-
					S^{\gamma}z\|\\
					& \leq   \omega_n\Big(\frac{\theta_1}{1+\rho \gamma} 
					(\theta_2\| z 
					-z^*  \| + \|D_{K_n} z- D_Kz\|)
					+\|(\alpha\id_{\G} + B)LS^{\gamma} z^* \|\Big)\\ 
					&\hspace{7cm}
					+ \frac{\kappa_K}{1+\rho \gamma} \|D_{K_n} z- D_Kz\|\\
					& \leq   \omega_n\left(\frac{\theta_1\theta_2}{1+\rho 
						\gamma}\| z 
					-z^*  \| + 
					\|(\alpha\id_{\G} + B)LS^{\gamma} z^* \|\right.\\ 
					&\hspace{4cm}\left.
					+ \frac{(\kappa_K+\theta_1\omega_n)}{1+\gamma 
						\rho}(\theta_1\|z-z^*\|+
					\|(\alpha\id_{\G} + B)L z^* \|)\right)\\
					&= \omega_n\left(
					\theta_1 \frac{\kappa_K+\theta_1 
						\omega_n+\theta_2}{1+\gamma \rho}
					\| z -z^*  \|
					+\|(\alpha\id_{\G} + B)LS^{\gamma} z^* \|\right.\\ 
					&\hspace{6cm}\left.+\frac{\kappa_K+\theta_1\omega_n}{1+\rho\gamma}\|(\alpha\id_{\G}
					+ B)Lz^*\|)\right).
				\end{align*}
				The result follows by defining 
				$(\theta_3,\theta'_3)=(
				\theta_1 (\kappa_K+\theta_1 \overline{\omega}+\theta_2),
				\kappa_K+\theta_1\overline{\omega}) $.
				\item It follows from \ref{prop:desop1}, 
				\ref{prop:desop2}, and 
				\ref{prop:desop4} that
				\begin{align*}
					&\| T_n^{\gamma} z - T^{\gamma} 
					z\|\\
					&= \|(\id_{\H}-\gamma D_{K_n})S_n^{\gamma} z+\gamma 
					D_{K_n}z - 
					(\id_{\H}-\gamma 
					D_K)S^{\gamma}z-\gamma D_K z\|\\
					& \leq \|S^{\gamma}_n z-S^{\gamma}z\|+\gamma \| 
					D_{K_n} 
					S^{\gamma}_n z - 
					D_KS^{\gamma}z\| + \gamma\|D_{K_n} z - D_Kz\|\\
					& \leq \left(\frac{1}{1+\gamma\rho}+\gamma\right)\| 
					D_{K_n}z -D_K z \|+\gamma \omega_n\left(\frac{\theta_3}{1+\gamma 
						\rho} \|z-z^*\|\right.\\ 
					&\hspace{3cm}\left.+\|(\alpha\id_{\G} + B)LS^{\gamma} z^* 
					\|+
					\frac{\theta'_3}{1+\gamma\rho} \|(\alpha\id_{\G} + 
					B)Lz^*\|\right)\\
					&\leq  \omega_n  \left(
					{\left(\Big(\frac{1}{1+\gamma\rho}+\gamma\Big)\theta_1+\frac{\gamma\theta_3}{1+\gamma
							\rho}\right)}
					\| z 
					-z^*  \|
					+\gamma\|(\alpha\id_{\G} + B)LS^{\gamma} z^* 
					\|\right.\\
					&\hspace{5.5cm}	\left.+\Big(\frac{1+\gamma 
						\theta'_3}{1+\gamma\rho}+\gamma\Big)
					\|(\alpha\id_{\G} + B)Lz^*\|\right).
				\end{align*}
				We conclude by defining 
				\begin{align*}
					\theta_4 &= 
					{((1+\chi+\chi^2|\rho|)\theta_1+\chi\theta)}\\
					\theta'_4 &= {\chi(1+\chi |\rho|)}\\
					\theta''_4 &= 1+{\chi(1+\theta'_3)+\chi^2 |\rho|}.
				\end{align*}
			\end{enumerate}
		\end{proof}
		\begin{teo}\label{teo:convergencia}
			In the context of Problem~\ref{prob:problem1} and  
			Assumption~\ref{assume:1}, let $(\gamma 
			_{n})_{n \in \N}$ be a sequence in
			$[\varepsilon,{\chi}-\varepsilon]$, for some $\varepsilon \in
			\left]0,\chi/2\right[$, 
			consider
			the sequence $(z_n)_{n \in \N}$ generated by 
			Algorithm~\ref{algo:BAD}.
			Then the following hold.
			\begin{enumerate}
				\item  $(z_n)_{n\in\N}$ converges 
				weakly to some solution to Problem~\ref{prob:problem1}.
				\item If $\hat{\rho}>0$ and there exists
				$\overline{\eta} \in [0,1[$ such that, for every $n\in 
				\N$, $\omega_n = \omega_0\, \overline{\eta}^n$, then 
				$(z_n)_{n\in \N}$
				converges linearly to the unique solution to
				Problem~\ref{prob:problem1}.
			\end{enumerate}
		\end{teo}
		\begin{proof}
			Let $z^* \in \zer (A+C+D_K)$ and, for every $\gamma \in 
			[\varepsilon,{\chi}-\varepsilon]$, consider the operators 
			$S^\gamma$, $T^\gamma$ and 
			$(S_n^\gamma)_{n\in \N}$, 
			$(T_n^\gamma)_{n\in \N}$,  
			defined in 
			\eqref{eq:defOp} and \eqref{eq:defOpn}, respectively. Then, 
			\eqref{e:algon} can be reexpressed as
			\begin{equation}\label{eq:defalgO}
				(\forall n \in \N) \quad x_n = S_n^{\gamma_n}z_n 
				\text{ and } z_{n+1}=T_n^{\gamma_n}z_n.
			\end{equation}
			\begin{enumerate}
				\item 
				In view of 
				Proposition~\ref{prop:BD}.\ref{prop:BD3}, 
				Proposition~\ref{prop:desop}.\ref{prop:desop5}, 
				and Lemma~\ref{lemma:QQn} 
				applied to 
				$I=[\varepsilon,{\chi}-\varepsilon]$,
				$S= \zer(A+C+D_K)$, 
				$Q^{\gamma}=T^{\gamma}$, 
				$\phi^{\gamma}: z 
				\mapsto \kappa_K^2\varepsilon^2\|z-S^{\gamma} z\|^2$, and
				\begin{equation}
					(\forall n \in \N) \quad
					\begin{cases}
						Q^{\gamma}_n=T_n^{\gamma}\\
						\varpi_n(z^*) = 
						\omega_n \upsilon\theta_4\\
						\eta_n(z^*) = \omega_n\upsilon(\theta'_4\|(\alpha 
						\id_{\G} + B)LS^{\gamma} z^* \| 
						+\theta''_4\|(\alpha \id_{\G} + B)Lz^*\|)
					\end{cases}
				\end{equation}
				with 
				\begin{equation}
					\upsilon = 
					\begin{cases}
						1 & \mbox{if $\rho \ge 0$}\\
						\frac{1}{1+\rho(\chi-\varepsilon)}
						& \mbox{if $\rho < 0$,}
					\end{cases}
				\end{equation} 
				$(\|z_{n}-z^*\|)_{n \in \N}$ 
				is convergent, $\sum_{n \in \N} 
				\|T_n^{\gamma_n}z_n-T^{\gamma_n}z_n\| < 
				+\infty$, and $\sum_{n\in \N} \|z_n-S^{\gamma_n} 
				z_{n}\|^2<+\infty$. Moreover, by 
				\eqref{eq:defalgO} and 
				Proposition~\ref{prop:desop}.\ref{prop:desop1}\&\ref{prop:desop2}	we 	obtain
				\begin{align*}
					(\forall n \in \N)\ \ 
					\|z_n-x_n\| &= \|z_n-S^{\gamma_n}z_n + 
					S^{\gamma_n}z_n 
					-S_n^{\gamma_n}z_n\|\\
					&\leq \|z_n-S^{\gamma_n}z_n\| + 
					\omega_n(\theta_1\|z_n-z^*\|
					+\|(\alpha\id_{\G} + B)Lz^*\|)\\
					&\leq \|z_n-S^{\gamma_n}z_n\| + 
					\omega_n(\theta_1\delta_z
					+\|(\alpha\id_{\G} + B)Lz^*\|),
				\end{align*}
				where 
				\begin{equation}\label{e:defdelta}
					\delta_z = \sup_{n\in \N} \|z_n-z^*\|<+\infty.
				\end{equation}
				
				Therefore 
				\begin{equation}\label{eq:zxto0}
					z_n-x_n\to 0.
				\end{equation}
				Furthermore, by  
				Proposition~\ref{prop:desop}.\ref{prop:desop1} and 
				the Lipschitzianity 
				of $D_K$, we have
				\begin{align*}
					\|D_{K_n}z_n-D_K x_n\| &\leq \|D_{K_n}z_n-D_K 
					z_n\|+\|D_K z_n-D_K x_n\|\\
					&\leq \omega_n(\theta_1\|z_n-z^*\|
					+\|(\alpha\id_{\G} + B)Lz^*\|)+\kappa_K \|z_n-x_n\|,
				\end{align*}
				hence
				\begin{equation}\label{eq:Bzxto0}
					D_{K_n} z_n-D_K x_n \to 0.
				\end{equation}
				Now, let $\overline{z}$ be a weak cluster point of 
				$(z_n)_{n \in \N}$ 
				and 
				let $(z_{k_n})_{n\in \N}$ be a subsequence such that 
				$z_{k_n} \weak 
				\overline{z}$. 
				It follows from \eqref{eq:zxto0} that 
				$z_{k_n}-x_{{k_n}}\to 0$ 
				and that $x_{k_n} \weak \overline{z}$ and from 
				\eqref{eq:Bzxto0} that 
				$D_{K_{k_n}} z_{k_n}-D_K x_{{k_n}}\to 0$. Moreover, the 
				cocoercivity of 
				$C$ yields $Cz_{k_n}-Cx_{k_n} \to 0$. 
				In addition, for every $n\in \N$,
				\begin{align}
					x_{k_n} = S_{k_n}^{\gamma_{k_n}}z_{k_n}\nonumber
					&\Leftrightarrow \quad 
					\frac{z_{k_n}-x_{k_n}}{\gamma_{k_n}}-(C+D_{K_{k_n}}) 
					z_{k_n}
					\in A x_{k_n}\nonumber\\
					&\Leftrightarrow \quad 
					\frac{z_{k_n}-x_{{k_n}}}{\gamma_{k_n}}- 
					(Cz_{k_n}-Cx_{{k_n}})-(D_{K_{k_n}} 
					z_{k_n}-D_Kx_{{k_n}}) \label{eq:limitsto0}\\
					& \hspace{6cm}\in 
					(A+C+D_K)x_{{k_n}}.\nonumber
				\end{align}
				Since $\{\gamma_n\}_{n 
					\in 
					\N} \subset [\varepsilon,{\chi}-\varepsilon]$, the 
				left-hand side converges strongly to 0 as $n\to 
				+\infty$.
				By the weak-strong closure of the maximally monotone
				operator $A+C+D_K$ (see 
				Proposition~\ref{prop:MP}.\ref{prop:MP3} \& 
				\cite[Proposition 20.38]{bauschkebook2017}), we conclude 
				that 
				$\overline{z} \in \zer(A+C+D_K)$. Finally, the weak 
				convergence of 
				$(z_n)_{n \in \N}$ to an element in  $\zer(A+C+D_K)$, 
				follows 
				from Lemma~\ref{lemma:QQn}.\ref{lemma:QQn4}. 
				\item The result follows from 
				Proposition~\ref{prop:BD}.\ref{prop:BD4}
				and Lemma~\ref{lemma:QQnlin} with
				$I=[\varepsilon,{\chi}-\varepsilon]$,
				$S= \zer(A+C+D_K)$, 
				$Q^\gamma = T^{\gamma}$, and
				\begin{align*}
					&\vartheta = 
					\sqrt{1-\varepsilon\min\{\kappa_K^2\varepsilon/2, 
						\hat{\rho}\}}\\
					& (\forall n\in \N)\quad Q_n^\gamma = T_n^\gamma\\
					&   (\forall n\in \N)\quad
					\varpi_n(z^*) = 
					\omega_n\upsilon \theta_4\\
					&\eta_0(z^*) = \omega_0\upsilon (\theta'_4\|(\alpha 
					\id_{\G} + B)LS^{\gamma} z^* \| 
					+\theta''_4\|(\alpha \id_{\G} + B)Lz^*\|).
				\end{align*}
			\end{enumerate}
		\end{proof}
		%
		
		\section{Forward-Douglas–Rachford-Forward Splitting}
		\label{se:FDRF}
		We will now turn our attention to the following algorithm.
		\begin{algo}\label{algo:RV} In the context of 
			Problem~\ref{prob:problem1}, let $\gamma \in \RPP$ be such 
			that $\gamma \rho > -1$, let
			$z_0 \in \H$, and 
			consider
			the iteration
			\begin{equation}
				\label{e:algonRV}
				(\forall n \in \N) \quad \left\lfloor
				\begin{aligned}
					&x_{n} =  J_{\gamma C}z_n\\
					&w_{n} =  D_{K_n} x_{n}\\
					&y_{n} =  J_{\gamma A}(2x_{n}-z_n-\gamma w_{n})\\
					&z_{n+1} = z_n + y_{n} - x_{n}-\gamma(D_{K_n}y_{n} - 
					w_{n}).
				\end{aligned}
				\right.
			\end{equation} 
		\end{algo}
		\begin{notation}
			In the context of Problem~\ref{prob:problem1}, for every 
			$\gamma 
			\in 
			\RP$ such that $\gamma \rho > -1$, define the 
			operators
			\begin{equation}\label{eq:defOpRV}
				R^\gamma=J_{\gamma 
					A} (2J_{\gamma C} -\id_{\H} -\gamma D_{K}  
				J_{\gamma C} ), 
				\quad 
				V^\gamma= 
				(\id_{\H}-\gamma D_{K}) R^\gamma  + \id_{\H} - 
				(\id_{\H}-\gamma D_{K})
				J_{\gamma 
					C}
			\end{equation}
			and, for every $n \in \N$,
			\begin{equation}\label{eq:defOpnRV}
				R_n^\gamma=J_{\gamma 
					A} (2J_{\gamma C} -\id_{\H} -\gamma D_{K_n}
				J_{\gamma C} ), \ \ 
				V^\gamma_n= 
				(\id_{\H}-\gamma D_{K_n}) R_n^\gamma  + \id_{\H} - 
				(\id_{\H}-\gamma D_{K_n}) 
				J_{\gamma C}.
			\end{equation}
			Additionally, define the set
			\begin{equation}
				\Gamma=\left\{\gamma \in \RPP
				\, \bigg|\,
				\kappa_{K}^2\gamma^2\left(
				1+\frac{\gamma}{2\beta}\right) <1\;
				\textnormal{and}\; \rho\gamma > -1
				\right\}.
			\end{equation}
		\end{notation}
		Note that $\Gamma \neq \varnothing$ since the involved conditions 
		are always satisfied for $\gamma$ small enough.

\begin{prop} \label{prop:RV}
	In the context of Problem~\ref{prob:problem1} and  
	Assumption~\ref{assume:1}, 
	let $\gamma \in \Gamma$, 
	let $\varepsilon_2 \in\, \RPP$ be such that 
	\begin{equation}\label{e:condesp2}
		\varepsilon_2 < 
		\frac{1-\kappa_K^2\gamma^2\Big(1+\frac{\gamma}{2\beta}\Big)}{1-\kappa^2_K\gamma^2},
	\end{equation}
	and  set $\varepsilon_1= 
	1-\kappa^{2}_K\gamma^2(1+\gamma/(2\beta(1-\varepsilon_2)))$.
	Then, the following assertions hold:
	\begin{enumerate}
		\item \label{prop:RV1} $ \zer (A+C+D_K) 
		= J_{\gamma C}(\fix 
		V^{\gamma})$.
		\item \label{prop:RV2}  For 
		every $z \in \H$ and every 
		$z^* 
		\in \fix V^{\gamma}$		
		\begin{align}\label{eq:propRV3}
			\|V^{\gamma} z- z^*\|^2 \leq  \|z- 
			z^*\|^2 
			-\varepsilon_1\|&J_{\gamma C} z-R^\gamma 
			z\|^2\\
			&-\frac{2\beta\varepsilon_2}{\gamma}\|J_{\gamma 
				C} 
			z-z+z^*-J_{\gamma C} z^*\|^2\nonumber.
		\end{align}
		\item \label{prop:RV4}  Suppose that 
		$\hat{\rho}
		>0$. 
		Then, for every 
		$z \in \H$ and every $z^* 
		\in \fix T^{\gamma}$, we have		
		\begin{align}\label{eq:propRV4}
			\|V^{\gamma} z- z^*\| \leq & 
			\sqrt{1-\frac{1}{3}\min\left\{\frac{2\beta\varepsilon_2}{\gamma},\varepsilon_1,2\gamma
				\hat{\rho}\right\}}
			\|z- z^*\|.
		\end{align}
	\end{enumerate}
\end{prop}
\begin{proof}\ 
	\begin{enumerate}
		\item See
		\cite[Lemma~4.1]{RyuVu2020}.
		\item \label{prop:RV2proof}	Let $z \in \H$ and	set $x=J_{\gamma 
			C} z$, $y= R^\gamma z$. Then, $2x-z-\gamma D_K x-y + 
		\gamma D_K y \in \gamma (A+D_K)y$.
		Let $z^* 
		\in \fix V^\gamma$ and set $(x^*,u)=(J_{\gamma C} 
		z^*,x-z+z^*-x^*)$.  Since $x^* 
		\in 
		\zer(A+C+D_K)$,  $x^*-z^* \in \gamma 
		(A+D_K)x^*$. From the monotonicity of $A+D_K$ established in 
		Proposition~\ref{prop:MP}.\ref{prop:MP1}, we deduce that
		\begin{align*}
			0&\leq  \scal{y-x^*}{2x-z-\gamma D_K x-y + 
				\gamma D_K y-x^*+z^*}\\
			&=  \scal{y-x^*}{x-\gamma D_Kx-y + 
				\gamma D_Ky}+\scal{y-x^*}{u}.
		\end{align*}
		Hence
		\begin{align}\label{eq:RV2proof}
			2\gamma \scal{y-x^*}{ D_Kx-D_Ky}	
			&\leq 2\scal{y-x^*}{x-y}+2\scal{y-x^*}{u}\nonumber\\
			&=\|x-x^*\|^2-\|y-x^*\|^2-\|x-y\|^2\\
			&\hspace{3cm}  +2\scal{y-x^*}{u}.\nonumber
		\end{align}
		We have then
		\begin{align}
			\|V^\gamma z - z^*\|^2&=
			\| (\id_{\H}-\gamma D_{K})y+z-
			(\id_{\H}-\gamma D_{K}) x-z^*\|^2\nonumber\\
			&= \| y-x^*+\gamma (D_{K}x-D_{K}y)-u\|^2
			\nonumber\\
			& \le \|x-x^*\|^2-\|x-y\|^2
			+\gamma^2 \|D_{K}x-D_{K}y\|^2 \nonumber \\
			& 
			\qquad - 2 \gamma \scal{D_{K}x-D_{K}y}{u} + \|u\|^2.
			\label{e:preVgazzstar}
		\end{align}
		As $C$ is $\beta$-cocoercive, it follows from 
		\cite[Lemma~3.2]{RyuVu2020} that
		\begin{equation}
			\|x-x^*\|^2=\|J_{\gamma C}z - J_{\gamma C} z^*\|^2
			\le \|z-z^*\|^2- \left(1+\frac{2\beta}{\gamma}\right) \|u\|^2.
		\end{equation}
		We deduce from this inequality and \eqref{e:preVgazzstar} that
		\begin{align}
			\|V^\gamma z - z^*\|^2
			& \le \|z-z^*\|^2-\|x-y\|^2
			+\gamma^2 \|D_{K}x-D_{K}y\|^2
			\nonumber\\
			&\hspace{4cm}- 2 \gamma \scal{D_{K}x-D_{K}y}{u} -\frac{2\beta}{\gamma} 
			\|u\|^2\nonumber\\
			& \le \|z-z^*\|^2-\|x-y\|^2-\frac{2\beta\varepsilon_2}{\gamma}
			\|u\|^2
			\\
			&\hspace{3cm}+\gamma^2\left(1+\frac{\gamma}{2\beta(1-\varepsilon_2)}
			\right)\|D_{K}x-D_{K}y\|^2.\nonumber
		\end{align}
		By using the fact that $D_K$ is $\kappa_K$-Lispchitzian, we get
		\begin{align}
			&\|V^\gamma z - z^*\|^2\nonumber\\
			& \le \|z-z^*\|^2-
			\left(1-\kappa_K^2\gamma^2\Big(1+\frac{\gamma}{2\beta(1-\varepsilon_2)}
			\Big)\right)\|x-y\|^2
			-\frac{2\beta\varepsilon_2}{\gamma}
			\|u\|^2,
		\end{align}
		which yields \eqref{eq:propRV3}.
		Condition \eqref{e:condesp2} can be satisfied since $\gamma \in 
		\Gamma$ and it guarantees that $\varepsilon_1> 0$.
		\item 
		By Proposition~\ref{prop:MP}.\ref{prop:MP2}, $A+D_K$ 
		is 
		strongly monotone. Hence, similarly to \ref{prop:RV2proof}, we 
		can show that
		\begin{align}\label{eq:RV3proof}
			2\gamma \hat{\rho}\|y-x^*\| + &2\gamma\scal{y-x^*}{ Dx- 
				Dy}	\nonumber\\
			&\leq\|x-x^*\|^2-\|y-x^*\|^2-\|x-y\|^2+2\scal{y-x^*}{u}.
		\end{align}
		and
		\begin{align}
			\|V^\gamma z- z^*\|^2
			&\leq  \|z- 
			z^*\|^2 
			-\varepsilon_1\|x-y\|^2-
			\frac{2\beta\varepsilon_2}{\gamma}\|x-z+z^*-x^*\|^2\nonumber\\
			&\hspace{7cm}          -2\gamma\hat{\rho}\|y-x^*\|^2\nonumber\\
			&\leq   \|z- 
			z^*\|^2 
			-\min\left\lbrace\frac{2\beta\varepsilon_2}{\gamma},\varepsilon_1,2\gamma
			\hat{\rho}\right\rbrace(\|x-y\|^2+\|y-x^*\|^2\nonumber\\
			&\hspace{6cm} +\|x-z+z^*-x^*\|^2)\nonumber\\
			&\leq   \|z- 
			z^*\|^2 - 
			\frac{1}{3}\min\left\lbrace\frac{2\beta\varepsilon_2}{\gamma},\varepsilon_1,2\gamma
			\hat{\rho}\right\rbrace\|z-z^*\|^2.
		\end{align}
	\end{enumerate}
\end{proof}

\begin{prop}\label{prop:desopRV}
	Consider the operators 
	defined by 
	\eqref{eq:defOpRV} and \eqref{eq:defOpnRV}.
	Let $\gamma \in \Gamma$.
	Then, there exists 
	$(\lambda_1,\lambda_2,\lambda_3,\lambda'_3,\lambda_4,\lambda'_4)\in 
	\RPP^6$ such that,
	for every 
	$(z,z^*) \in 
	\H^2$, for every $\gamma \in \Gamma$,
	and for every $n\in \N$,
	the following inequalities hold:
	\begin{enumerate}
		\item \label{prop:desop1RV}  
			$\displaystyle\|R_n^{\gamma} z - R^{\gamma}z\|\leq 
			\omega_n\left(\lambda_1\|z-z^*\|
			+\frac{\gamma}{1+\rho \gamma}\|(\alpha \id_{\G} + 
			B)LJ_{\gamma C}z^*\|\right)$
		\item\label{prop:desop2RV}	 
			$\displaystyle 
			\|R_n^\gamma z-R^\gamma z^*\| \leq\lambda_2\|z-z^*\| 
			+\frac{\omega_n\gamma}{1+\rho\gamma} \|(\alpha \id_{\G} + 
			B)LJ_{\gamma C}z^*\|$
		\item \label{prop:desop3RV} 
			$ \|D_{K_n}R_n^\gamma z- D_K R^\gamma z \| \leq 
			\omega_n(\lambda_3\|z-z^*\|
			+\lambda'_3\|(\alpha \id_{\G} + B)LJ_{\gamma 
				C}z^*\|$\\
			\hspace*{\fill}$+\|(\alpha \id_{\G} + B)LR^\gamma z^*\|)$
		\item \label{prop:desop4RV} 
		$
		\|V_n^\gamma z- V^\gamma z\|\leq \omega_n(\lambda_4 \|z-z^*\|
		+\lambda'_4\|(\alpha \id_{\G} + B)LJ_{\gamma 
			C}z^*\|$\\
		\hspace*{\fill}$+\gamma\|(\alpha \id_{\G} + B)LR^{\gamma}z^*\|)$.
\end{enumerate}
\end{prop}
\begin{proof}	Recall that, in view of 
Assumption~\ref{assume:1}, 
$\overline{\omega}=\sup_{n\in \N} \omega_n <+\infty$. Let 
$(z,z^*)\in \H^2$.
\begin{enumerate}
	\item It follows from the $(1+\gamma \rho)^{-1}$-Lipschitzianity 
	of $J_{\gamma A}$, the nonexpansiveness of $J_{\gamma C}$, 
	and 	Proposition~\ref{prop:desop}.\ref{prop:desop1} that
	\begin{align}
		\|R_n^\gamma z - R^\gamma z\| &= \| J_{\gamma 
			A} (2J_{\gamma C} -\id_{\H} -\gamma D_{K_n} \circ 
		J_{\gamma C} 
		)z\nonumber\\
		&\hspace{4cm}-J_{\gamma 
			A} (2J_{\gamma C} -\id_{\H} -\gamma D_K \circ J_{\gamma 
			C} 
		)z\|\nonumber\\
		&\leq \frac{\gamma}{1+\rho\gamma}\|D_{K_n}J_{\gamma C}z - D_K 
		J_{\gamma C}z\|\nonumber\\
		& \leq 
		\frac{\gamma\omega_n}{1+\rho\gamma}(\theta_1\|J_{\gamma 
			C}z-J_{\gamma C}z^*\|
		+\|(\alpha \id_{\G} + B)LJ_{\gamma C}z^*\|)\nonumber\\
		& \leq \frac{\gamma\omega_n}{1+\rho\gamma}(\theta_1\|z-z^*\|
		+\|(\alpha \id_{\G} + B)LJ_{\gamma C}z^*\|).
	\end{align}
	The result follows by setting  
	$\lambda_1=\gamma\theta_1/(1+\rho \gamma)$,
	$\theta_1$ being given by \eqref{e:deftheta1}.
	\item Using the nonexpansiveness of $J_{\gamma A}$, the 
	Lipschitzianity of $D_K$, and the 
	nonexpansiveness of $2J_{\gamma C}-\id_H$  
	\cite[Corollary~23.11]{bauschkebook2017}, 
	we deduce from \ref{prop:desop1RV} that
	\begin{align}
		\|R_n^\gamma z&-R^\gamma z^*\|\nonumber\\
		&\leq  \| R^\gamma z-R^\gamma 
		z^*\|+\| R_n^\gamma z-R^\gamma z\|\nonumber\\
		&= \frac{1}{1+\gamma \rho}\| (2J_{\gamma C} -\id_{\H} -\gamma 
		D_K J_{\gamma C} 
		)z\nonumber\\
		&\hspace{3cm}-(2J_{\gamma 
			C} -\id_{\H} -\gamma D_K J_{\gamma C} )z^*\|+\| 
		R_n^\gamma z-R^\gamma 
		z\|\nonumber\\
		& \leq \Big(\frac{1+\gamma\kappa_K}{1+\gamma 
			\rho}+\omega_n\lambda_1\Big)\|z-z^*\| 
		+\frac{\omega_n\gamma}{1+\gamma \rho}\|(\alpha \id_{\G} + 
		B)LJ_{\gamma C}z^*\|.
	\end{align}
	We conclude by defining 
	\begin{equation}
		{\lambda_2=\overline{\omega}\lambda_1+\frac{\gamma\kappa_K+1}{1+\gamma
				\rho}}.
	\end{equation}
	\item It follows from \ref{prop:RV1}, \ref{prop:RV2}, and 
	Proposition~\ref{prop:desop}.\ref{prop:desop1} 
	that
	\begin{align}
		&\|D_{K_n}R_n^\gamma z- D_K R^\gamma z \|\nonumber\\
		&\leq \|D_{K_n}R_n^\gamma z- 
		D_K R_n^\gamma z\| + \|D_K R_n^\gamma z- D_K R^\gamma z 
		\|\nonumber\\
		&\leq \omega_n(\theta_1\|R_n^\gamma z-R^\gamma z^*\|
		+\|(\alpha \id_{\G} + B)LR^\gamma z^*\|) + 
		\kappa_K\|R_n^\gamma z-R^\gamma z\|\nonumber\\
		&\leq \omega_n\Big(
		(\theta_1\lambda_2+\kappa_K\lambda_1)
		\|z-z^*\|
		+\frac{\gamma(\omega_n\theta_1+\kappa_K)}{1+\gamma 
			\rho}\|(\alpha \id_{\G} + B)LJ_{\gamma 
			C}z^*\|\nonumber\\
		&\hspace{7cm}+
		\|(\alpha \id_{\G} + B)LR^\gamma z^*\|\Big).\nonumber
	\end{align}
	The result is obtained by defining 
	$\lambda_3 = (\theta_1\lambda_2+\kappa_K\lambda_1)$
	and  $\lambda'_3 = \gamma 
	(\overline{\omega}\lambda_1+\kappa_K)/(1+\gamma\rho)$.
	
	\item It follows from \ref{prop:desop1RV}, \ref{prop:desop3RV}, 
	Proposition~\ref{prop:desop}.\ref{prop:desop1}, and the 
	nonexpansiveness of $J_{\gamma C}$ that
	\begin{align}
		&\|V_n^\gamma z- V^\gamma z\| \nonumber\\
		&\leq \| R_n^\gamma z-R^\gamma 
		z\|+\gamma \|D_{K_n}R_n^\gamma z- D_K R^\gamma z\|+ 
		\gamma\|D_{K_n}J_{\gamma 
			C}z-D_K J_{\gamma C}z\|\nonumber\\
		&\leq 
		\omega_n\Bigg((\lambda_1+\gamma(\lambda_3+\theta_1))\|z-z^*\|
		\nonumber\\
		&\hspace{0.5cm}+\gamma
		{\Big(\lambda'_3+1+\frac{1}{1+\gamma \rho}\Big)}
		\|(\alpha \id_{\G} + B)LJ_{\gamma 
			C}z^*\|
		+\gamma\|(\alpha \id_{\G} + 
		B)LR^{\gamma}z^*\|\Bigg).\nonumber
	\end{align}
	This yields the sought inequality
	by defining
	{$\lambda_4 = \lambda_1+\gamma(\lambda_3+\theta_1)$
		and 
		\begin{equation}
			\lambda'_4= \lambda'_3+1+\frac{1}{1+\gamma \rho}.
		\end{equation}
	}
\end{enumerate}
\end{proof}

\begin{teo}\label{teo:convergenciaRV}
In the context of Problem~\ref{prob:problem1} and  
Assumption~\ref{assume:1}, let $\gamma \in \Gamma$, 
and consider
the sequences $(z_n)_{n \in \N}$ and $(x_n)_{n\in \N}$ generated by 
Algorithm~\ref{algo:RV}.
Then the following hold.
\begin{enumerate}
	\item  $(z_n)_{n\in\N}$ converges 
	weakly to some $\overline{z}\in	\fix V^{\gamma}$ and 
	$(x_n)_{n\in\N}$ converges weakly to $J_{\gamma 
		C}\overline{z} \in \zer(A+C+D_K)$.
	\item If $\hat{\rho}>0$ and there exists
	$\overline{\eta} \in [0,1[$ such that, for every $n\in \N$, 
	$\omega_n = \omega_0\, \overline{\eta}^n$, then $(z_n)_{n\in \N}$
	converges linearly to $\overline{z}\in	\fix V^{\gamma}$ and 
	$(x_n)_{n\in\N}$ converges linearly to $J_{\gamma 
		C}\overline{z}$, which is 
	the unique solution to
	Problem~\ref{prob:problem1}.
\end{enumerate}
\end{teo}
\begin{proof}
Let $\gamma \in \Gamma$.
Consider 
the operators $R^\gamma$, $V^\gamma$ and $(R_n^\gamma)_{n\in \N}$, 
$(V_n^\gamma)_{n\in \N}$  
defined in 
\eqref{eq:defOpRV} and \eqref{eq:defOpnRV}, respectively. 
Let $x^* \in \zer (A+C+D_K)$.
According to Proposition~\ref{prop:RV}.\ref{prop:RV1}, there exists
$z^*\in \fix V^\gamma$ such that 
$x^*=J_{\gamma C}z^*$. Note that 
\eqref{e:algonRV} is equivalent to
\begin{equation}\label{eq:defalgORV}
	(\forall n \in \N) \quad y_n = R_n^{\gamma}z_n 
	\text{ and } z_{n+1}=V_n^{\gamma}z_n.
\end{equation}
\begin{enumerate}
	\item
	In view of Proposition~\ref{prop:RV}.\ref{prop:RV2} and 
	Proposition~\ref{prop:desopRV}.\ref{prop:desop4RV}, 
	Lemma~\ref{lemma:QQn} can be 
	applied to $I= \{\gamma\}$,
	$S = \fix V^\gamma$, $Q^{\gamma}=V^{\gamma}$, 
	$\phi^{\gamma}: z 
	\mapsto \varepsilon_1\|J_{\gamma C} z-R^\gamma 
	z\|^2+\frac{2\beta\varepsilon_2}{\gamma}\|J_{\gamma 
		C} 
	z-z+z^*-J_{\gamma C} z^*\|^2$, and
	\begin{align*}
		(\forall n &\in \N)\\ 
		&\begin{cases}
			Q^{\gamma}_n=V_n^{\gamma}\\
			\varpi_n(z^*)=\omega_n \lambda_4\\
			\eta_n(z^*) = \omega_n(\lambda'_4\|(\alpha \id_{\G} + 
			B)LJ_{\gamma 
				C}z^*\|+\gamma\|(\alpha \id_{\G} + 
			B)LR^{\gamma}z^*\|).
		\end{cases}
	\end{align*}
	This allows us to deduce that
	$(\|z_{n}-z^*\|)_{n \in \N}$ 
	is convergent, $\sum_{n \in \N} 
	\|V_n^{\gamma}z_n-V^{\gamma}z_n\| < 
	+\infty$, $\sum_{n\in \N} \|x_n-R^{\gamma} z_{n}\| < +\infty $, 
	and
	$\sum_{n\in \N} \|x_n-z_n-x^*+z^*\| < +\infty$.
	Moreover, according to \eqref{eq:defalgORV} and 
	Proposition~\ref{prop:desopRV}.\ref{prop:desop1RV},
	\begin{align*}
		\|x_n-y_n\| &= \|x_n-R^{\gamma}z_n + R^{\gamma}z_n 
		-R_n^{\gamma}z_n\|\\
		&\leq \|x_n-R^{\gamma}z_n\| + 
		\omega_n\left(\lambda_1\|z_n-z^*\|\frac{\gamma}{1+\gamma \rho}\|(\alpha \id_{\G} + 
		B)LJ_{\gamma C}z^*\|\right)\\
		&\leq \|x_n-R^{\gamma}z_n\| + \omega_n\left(\lambda_1\delta_z
		+\frac{\gamma}{1+\gamma \rho}\|(\alpha \id_{\G} + 
		B)LJ_{\gamma C}z^*\|\right),
	\end{align*}
	where $\delta_z$ is given by \eqref{e:defdelta}.
	Therefore,
	\begin{equation}\label{eq:zxto0RV}
		y_n-x_n\to 0
	\end{equation}
	and, it follows from the cocoercivity of $C$ and the Lipschitzian 
	property of $D$ that
	\begin{equation}\label{eq:Czxto0RV}
		Cy_{n}-Cx_{n} \to 0 \text{ and } D_K y_{n}-D_K x_{n} \to 0.
	\end{equation}
	Since $z_n-x_n=\gamma Cx_n$, we deduce that
	\begin{equation}\label{eq:Czxto0RV2}
		{\frac{z_{n}-x_{n}}{\gamma}}
		-Cy_{n} \to 0.
	\end{equation}
	Furthermore, by Proposition~\ref{prop:desop}.\ref{prop:desop1} and 
	the nonexpansiveness of $J_{\gamma C}$ we have
	\begin{align*}
		\|w_n-D_K y_n\|  &\leq 
		\|D_{K_n}x_n-D_Kx_n\|+\|D_Ky_n-D_Kx_n\|\\
		&\leq \omega_n(\theta_1\|z_n-z^*\|
		+\|(\alpha\id_{\G} + B)LJ_{\gamma C} z^*\|)+\kappa_K 
		\|y_n-x_n\|.
	\end{align*}
	Thus
	\begin{equation}\label{eq:Bzxto0RV}
		w_n-D_K y_n \to 0.
	\end{equation}
	Now, let $\overline{z}$ be a weak cluster point of $(z_n)_{n \in 
		\N}$ 
	and 
	let $(z_{k_n})_{n\in \N}$ be a subsequence such that $z_{k_n} 
	\weak 
	\overline{z}$.
	Since $x_{k_n}-z_{k_n} \to x^*-z^*$,
	\begin{equation}\label{e:weakxn}
		x_{k_n}\weak \overline{x} = \overline{z}+x^*-z^*.
	\end{equation}
	According to \eqref{eq:zxto0RV}, $x_{k_n}-y_{{k_n}}\to 0$, hence 
	that {$y_{k_n}\weak \overline{x}$}.
	It follows from \eqref{eq:Czxto0RV},
	\eqref{eq:Czxto0RV2}, and \eqref{eq:Bzxto0RV}
	that
	$Dx_{k_n}-Dy_{k_n}\to 0$, $(x_{k_n}-z_{{k_n}})/\gamma
	-Cy_{k_n} \to 0$, and $D_K y_{{k_n}}-w_{{k_n}} 
	\to 0$.
	Furthermore, from \eqref{e:algonRV},
	\begin{align}\label{eq:limitsto0RV}
		&\frac{2 x_{k_n}-z_{k_n}-y_{k_n}}{\gamma}- w_{k_n}\in A 
		y_{k_n}\nonumber\\
		\Leftrightarrow\;\; &	
		\frac{x_{k_n}-y_{{k_n}}}{\gamma}-\left({\frac{z_{k_n}-x_{k_n}}{\gamma}}
		-Cy_{k_n}\right)-(D_Kx_{k_n}-D_Ky_{k_n})\nonumber\\
		& \hspace{5cm}+(D_Kx_{k_n}-w_{k_n})\in 
		(A+C+D_K)y_{k_n}.\nonumber
	\end{align}
	Altogether, by the weak-strong closure of the maximally monotone
	operator $A+C+D_K$ (see Proposition~\ref{prop:MP}.\ref{prop:MP3} 
	\& 
	\cite[Proposition 20.38]{bauschkebook2017}), we conclude that 
	${\overline{x}} \in \zer(A+C+D_K)$. 
	We can thus choose $x^* = \overline{x}$
	and \eqref{e:weakxn} yields $\overline{z}
	= z^* = J_{\gamma C} \overline{x} \in \fix V^\gamma$.
	The weak convergence of $(z_n)_{n \in 
		\N}$ to $\overline{z}$ follows 
	from Lemma~\ref{lemma:QQn}.\ref{lemma:QQn4}.
	Finally $x_n-z_n \to \overline{x}-\overline{z}$
	$\Rightarrow$
	$x_n \weakly \overline{x}$.
	\item The linear convergence of $(z_n)_{n\in \N}$ to $z^*$ 
	follows from Proposition~\ref{prop:RV}.\ref{prop:RV2}
	and Lemma~\ref{lemma:QQnlin} with
	$I=\{\gamma\}$,
	$S = \fix V^\gamma$,
	$Q^\gamma = V^{\gamma}$, and
	\begin{align}
		&\vartheta = 
		\sqrt{1-{\frac{1}{3}}\min\left\{\frac{2\beta\varepsilon_2}{\gamma},\varepsilon_1,2\gamma
			\hat{\rho}\right\}}\\
		& (\forall n\in \N)\quad Q_n^\gamma = V_n^\gamma\\
		&   (\forall n\in \N)\quad
		\varpi_n(z^*)=\omega_n \lambda_4\\
		&\eta_0(z^*) = \omega_0(\lambda'_4\|(\alpha \id_{\G} + 
		B)LJ_{\gamma 
			C}z^*\|+\gamma\|(\alpha \id_{\G} + B)LR^{\gamma}z^*\|).
	\end{align}
	Since $J_{\gamma C}$ is nonexpansive,
	\[
	(\forall n \in \N)\quad \|x_n-x^*\| \le \|z_n -z^*\|.
	\]
	This shows that $(x_n)_{n\in \N}$ converges linearly to $x^*$, 
	which is the unique solution to Problem~\ref{prob:problem1}.
\end{enumerate}
\end{proof}

\section{Numerical Experiments} \label{se:numexp}
This section is devoted to illustrate our theoretical results, 
through numerical experiments on an image reconstruction problem 
arising in Computed Tomography (CT), in material science. 

\subsection{Problem formulation and settings}
In CT~\cite{Kak2001}, one aims at solving the inverse problem of 
retrieving an estimate of a sought image $\bar{x} \in \mathbb{R}^N$, 
with $N \geq 1$ pixels, from acquisitions
\begin{equation}\label{eq:model}
c = \mathcal{D}(L \bar{x}),
\end{equation}
where  $L \in \R^{N \times M}$ is a forward linear operator acting as 
a discretized Radon projector, $\mathcal{D}: \mathbb{R}^M \to 
\mathbb{R}^M$ models some noise perturbing the acquisitions, and $c 
\in \mathbb{R}^M$  is the noisy tomographic projection. We focus on 
the challenging situation when the back-projector matrix $L^\top 
\colon \R^M \to \R^N$ is approximated by $K \colon \R^M \to \R^N$. 
This is a current situation in practical CT reconstruction, where 
operator $L$ (and thus, its transpose) cannot be stored, for memory 
reasons. It is instead implemented as a function, which computes 
on-the-fly projection and back-projections operations, making use of 
fast operations involving advanced interpolation strategies 
\cite{Xu06}. The adjoint mismatch is thus inherent to this 
application~\cite{Hansen2022,Zeng2018} and, except in special 
simplistic cases, cannot be avoided. 

An efficient approach to retrieve an estimate $\bar{x}$ from $c$, 
$L$, and $K$, consists of minimizing
a penalized cost function, in the form of 
Problem~\ref{prob:problem0o}. However, as explained earlier, due to 
the adjoint mismatch, the formulation in Problem~\ref{prob:problem0o} 
is not well suited, and we propose instead to solve the following 
mismatched monotone inclusion:
\begin{equation}\label{eq:mmop}
\text{find } x \in \R^N \text{ such that } 0 \in \partial_{\rm F} 
f(x) + \nabla g (x) + \alpha K(Lx-c)+K\nabla h (Lx), 
\end{equation}
with $f\colon \R^N \to \R$ and $g \colon \R^N \to \R$ playing the 
role of regularization terms favoring a priori properties on the 
estimated image, and $h \circ L$ the data fidelity term accounting 
for the noise model. 
The latter inclusion problem reads as a particular instance of 
Problem~\ref{prob:problem1}, by setting $A = \partial_{\rm F} f $, $B 
= \nabla h $, and $C = \nabla g$ under suitable assumptions on the 
involved functions. In particular, we will choose $f$ and $g$ so that 
$\hat{\rho}>0$ and the inclusion in \eqref{eq:mmop} has a unique 
solution (Proposition~\ref{prop:MP}.\ref{prop:MP3}).

\textbf{Data fidelity term:} 
We consider a general mixed multiplicative/additive noise model, as 
discussed for instance in \cite{ChouzenouxPG}. The vector $c$ is 
related to $\bar{x}$ through
\begin{equation}
c = z + e,
\end{equation}
with $z | \bar{x} \sim \mathcal{P}(L \bar{x})$ (i.e., Poisson 
distribution with mean $L \bar{x}$) and $e \sim 
\mathcal{N}(0,\sigma^2 \mathrm{Id})$ (i.e., i.i.d. Gaussian 
distribution with zero-mean and variance $\sigma^2$). Such a noise 
model allows to both account for multiplicative noise typical from 
emission tomography scenarios, and additive noise coming from the 
sensors. As shown in \cite{Zanni2015,MarnissiTCI}, a suitable choice 
for the data fidelity term in such case is the Generalized Anscombe 
function, which is a smoothed approximation of the neg-log-likelihood 
associated to a Gauss-Poisson noise model. Under the assumption that 
$c = (c_m)_{1 \leq m \leq M} \in [- \frac{3}{8} - 
\sigma^2,+\infty[^M$ (which can be satisfied by basic cropping), 
function $h$ reads
\begin{equation}
(\forall y = (y_m)_{1\le m \le M} \in \R^M) \quad h(y) = \sum_{m=1}^M 
\varphi(y_m;c_m), 
\end{equation}
where, for every $a \in \R$, and every $b\in [- \frac{3}{8} - 
\sigma^2,+\infty[$,
\begin{equation}
\varphi(a;b) = 
\begin{cases}
	2 \left( \sqrt{b + \frac{3}{8} + \sigma^2} - \sqrt{a + 
		\frac{3}{8} + \sigma^2} \right)^2 &\text{if} \quad a \geq 0,\\
	\varphi(0;b) + \dot{\varphi}(0;b) a + \frac{1}{2} \nu(b) a^2 & 
	\text{otherwise},
\end{cases}
\end{equation}
with 
\begin{equation}
\left(\forall b \geq - \frac{3}{8} - \sigma^2\right) \quad \nu(b) = 
\left(\frac{3}{8} + \sigma^2\right)^{-\frac{3}{2}} \sqrt{\frac{3}{8} 
	+  b + \sigma^2}.
	\end{equation}
	Basic calculus shows that, for every $b \geq - \frac{3}{8} - 
	\sigma^2$, the derivative of
	$\varphi(\cdot;b)$ at $a \geq 0$ reads
	\begin{equation}
\dot{\varphi}(a;b) = 2 - \frac{2 \sqrt{8 b + 8 \sigma^2 + 3}}{\sqrt{8 
		a + 8 \sigma^2 + 3}}.
		\end{equation}
	Under this definition, we can readily show that, for every $b \geq - 
	\frac{3}{8} - \sigma^2$, $\dot{\varphi}(\cdot\, ; b)$ is Lipschitzian 
	on $\mathbb{R}$, with constant
	$\nu(b)$. Assuming that the observed data satisfies $c \in [- 
	\frac{3}{8} - \sigma^2,+\infty[^M$, we deduce  that $h $ is 
	$\zeta$-Lipschitz differentiable on $\mathbb{R}^N$ with
	\begin{equation}\label{eq:LC}
\zeta   = \max_{m \in \{1,\ldots,M\}} \nu(c_m).
\end{equation}

\textbf{Regularization terms:} 
Function $f$ imposes the range of the restored image and controls the 
image energy, and is defined as
\begin{equation}
(\forall x \in \R^N) \quad f(x) = \iota_{[0,x_{\max}]^N}(x) + 
\frac{\rho}{2} \| x \|^2
\end{equation}
with $\rho \in ]0,+\infty[$.
Function $f$ 
is $\rho$-strongly convex on $\R^N$. Its proximity operator has the 
following closed form expression:
\begin{equation}
(\forall \gamma\in \RPP) \quad \prox_{\gamma f}(x) = \min\{ \max \{ 
(\gamma \rho + 1)^{-1} x , 0 \} , x_{\max}\}.
\end{equation}
Function $g$ promotes sparsity of the image in a transformed domain 
defined by a linear operator $W \in \R^{N \times N}$:
\begin{equation}
(\forall x \in \R^N) \quad g(x) =  (\Phi_{\delta} \circ W) (x).
\end{equation}
Hereabove,  $\Phi_\delta$ is the Huber function defined, for 
$\delta>0$, as
\begin{equation}\label{eq:def_huber}
(\forall x = (x_i)_{1\leq i \leq N} \in \R^N) \quad \Phi_{\delta}(x) 
=  \sum_{i=1}^N \phi_{\delta}(x_i), \quad  
\end{equation}
with
\begin{equation}
(\forall \eta \in \R) \quad
\phi_{\delta} (\eta) = \begin{cases}
	|\eta|-\frac{\delta}{2}, & \text{ if } |\eta| > \delta,\\
	\frac{\eta^2}{2\delta}, & \text{ otherwise}.
\end{cases}
\end{equation}
Function $\Phi_\delta$ can be viewed as a smoothed approximation of 
the $\ell_1$ penalty, promoting the sparsity of its argument. 
Function $g$ belongs to $ \Gamma_0(\R^N)$. Moreover, 
the derivative of $\phi_\delta$ reads
\begin{align}
(\forall \eta \in \R) \quad \dot{\phi}_{\delta}(\eta) =
\begin{cases}
	\frac{|\eta|}{\eta}, & \text{ if } |\eta| > \delta,\\
	\frac{\eta}{\delta}, & \text{otherwise},
\end{cases}
\end{align}
which shows that $\Phi_{\delta}$ has $(1/\delta)$-Lipschitizian 
gradient. We set $W\in \R^{N \times N}$ as an orthonormal wavelet 
transform~\cite{PustelnikWavelet}, that leads to efficient penalties 
in tomography~\cite{Klann_2015,Loris2007}. Then $\|W\|=1$ and $g$ 
also has $(1/\delta)$-Lipschitizian gradient. Additionally, by 
orthogonality of $W$, \cite[Corollary~23.27]{bauschkebook2017} yields
\begin{equation}
(\forall \gamma \in \RPP) \quad \prox_{\gamma g} = W^\top \circ 
\prox_{\gamma \Phi_\delta} \circ W,
\end{equation}
with $W^\top = W^{-1}$, and, by 
\cite[Proposition~24.11]{bauschkebook2011},
\begin{align}
(\forall \gamma\in \RPP) (\forall x = (x_i)_{1\leq i \leq N} \in 
\R^N) \quad \prox_{\gamma \Phi_\delta}(x ) = \big(\prox_{\gamma 
	\phi_\delta}(x_i)\big)_{1\leq i \leq N},
	\end{align}
	with
	\begin{align}
(\forall \gamma \in \RPP) (\forall \eta \in \R) \quad \prox_{\gamma 
	\phi_\delta}(\eta) =  \begin{cases}
	\eta - \frac{\gamma|\eta|}{\eta}, & \text{ if } |\eta| > 
	\delta+\gamma,\\
	\frac{\delta\eta}{\gamma+\delta}, & \text{ if } |\eta| \leq 
	\delta+\gamma.
\end{cases}
\end{align}

\textbf{Algorithms implementation:}
We are now ready to apply Algorithm~\ref{algo:BAD} (MMFBHF)
and Algorithm~\ref{algo:RV} (MMFDRF) to solve Problem~\ref{prob:problem0o} (MM stands for 
MisMatched).
In the considered setting, the algorithms read as follows. 
\begin{algo}[MMFBHF]
Let $\gamma>0$, let $z_0 \in \mathbb{R}^N$, and consider the iteration
\begin{equation}
	(\forall n \in \mathbb{N}) \quad \left\lfloor
	\begin{aligned}
		&u_n =   K(\alpha \id  +\nabla h)( L z_{n})-\alpha K c\\
		&y_n = z_n-\gamma
		(\nabla g (z_n) + u_n)\\
		&x_n =\prox_{\gamma f}(y_n)\\
		&z_{n+1} = x_n + \gamma (u_n - 
		K(\alpha \id  +\nabla h)( L x_{n})+\alpha K c ).
	\end{aligned}
	\right.
\end{equation} 
\end{algo}
\begin{algo}[MMFDRF] Let $\gamma>0$, let $z_0 \in \mathbb{R}^N$, and 
consider the iteration
\begin{equation}
	(\forall n \in \mathbb{N}) \quad \left\lfloor
	\begin{aligned}
		&x_{n} =  \prox_{\gamma g}z_n\\
		&w_{n} =   K(\alpha \id  +\nabla h)( L x_{n})-\alpha K c\\
		&y_{n} =  \prox_{\gamma f}(2x_{n}-z_n-\gamma w_{n})\\
		&z_{n+1} = z_n + y_{n} - x_{n}-\gamma(K(\alpha \id  +\nabla 
		h)( L y_{n})-\alpha K c - 
		w_{n}).
	\end{aligned}
	\right.
\end{equation} 
\end{algo}

The projector $L$ is given by the line length ray-driven projector 
\cite{zengRayDriven93} and implemented in MATLAB using the line 
fan-beam projector provided by the ASTRA toolbox 
\cite{vanAarleASTRA2016,vanAarleASTRA2015}. Moreover, a constant 
mismatch, i.e., for every $n \in \N$, $K_n = K$, is considered, where 
the mismatched backprojector $K$ is the adjoint of the strip fan-beam 
projector from the ASTRA toolbox. 

In order to set up the stepsize parameters guaranteeing the 
convergence of our algorithms, we need to evaluate $\lambda_{\min}$, 
defined in \eqref{eq:deflmin}. To do so, we compute the eigenvalues 
of the operator $(KL + L^{\top} K^{\top})/2$ by using the function 
{\it eigs} from MATLAB, yielding $\lambda_{\min}\approx -6.0082$. 
Note that, it would also be possible to estimate $\lambda_{\min}$ 
avoiding an explicit implementation of $L^\top$ and $K^\top$ by the 
strategy proposed in \cite{UnmatchedDong2019}. In order to guarantee 
that Assumption~\ref{assume:1}.\ref{eq:desassume1} holds, we set 
$\rho=-\alpha\lambda_{\min}+\widetilde{\zeta}_{L^\top-K}+10^{-3}$, 
where $\widetilde{\zeta}_{L^\top-K}$ is estimated as 
$\|L^\top-K\|\|L\|\zeta$. The spectral norms $\|L^\top-K\|$ and 
$\|L\|$ are computed using the power iterative method. We implement 
MMFBHF with constant step-size $\gamma = 
({3.99{\beta}})/({1+\sqrt{1+16{\beta}^2
	\kappa_K^2}})$ and MMFDRF with $\gamma = 0.999\hat{\gamma}$, 
	where $\hat{\gamma}$ is the largest solution to the equation 
	$\kappa_K^2\hat{\gamma}^2(1+\hat{\gamma}/(2\beta)))=1$, computed 
	numerically.
	These choices allow satisfying our technical assumptions, so that 
	the convergence theorems hold. 
	
	\subsection{Experimental results}
	
	We now present our experimental results. 
	In the observation model \eqref{eq:model}, the ground truth image 
	$\overline{x}$ represents a part of a high resolution scan of a 
	phase-separated barium borosilicate glass imaged at the ESRF 
	synchrotron~\cite{ChouzenouxTomo}\footnote{https://www.esrf.fr/ - 
The dataset is a courtesy of David Bouttes.}. The image size is 
$N = 128 \times 128$ pixels. The projector $L$ describes a 
fan-beam geometry over 180$^o$ using $90$ regularly spaced 
angular steps. The source-to-object distance is $800$ mm, and the 
source-to-image distance is $1200$ mm. The bin grid is twice 
upsampled with respect to the pixel grid,  the detector has
$249$ bins of size $1.6$ mm, so that $M = 90 \times 249$. The 
pixel values of $\overline{x}$  outside a circle of diameter 
$128$ pixels are set to $0$, to guarantee that the object of 
interest lies within the field of view.  
\begin{figure}
\centering
\subfloat[\scriptsize Tomographic projection $L 
\bar{x}$]{\label{fig:sino_wav}\includegraphics[scale=0.22]{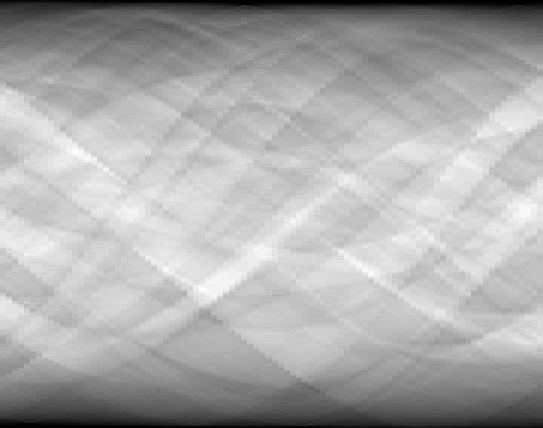}}\,
\subfloat[\scriptsize Noisy projection 
$c$.]{\label{fig:sinonoise_wav}\includegraphics[scale=0.22]{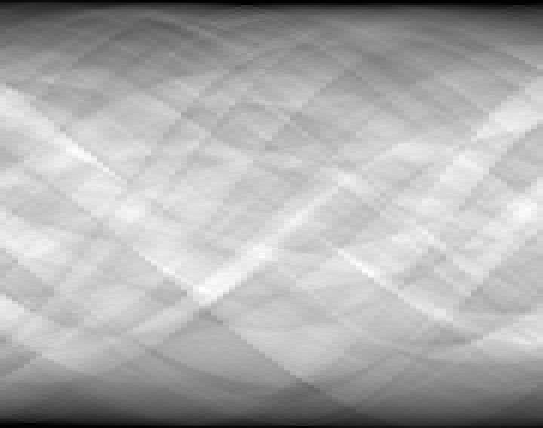}}
\caption{Clean projection and its noisy version, with 
	$\textnormal{SNR}_{\textnormal{input}}=42.18 $ dB} 
\label{fig:sinogram_wav}
\end{figure}

The image intensity range lies in $[0, x_{\max}]$, with $x_{\max} = 
900$. The Gaussian noise level is set to $\sigma=200$. The input 
signal-to-noise-ratio (SNR) in decibels (dB), between the clean 
projection $L \bar{x}$ and $c$ (both displayed on Figure 
\ref{fig:sinogram_wav}) is defined as
\begin{equation}
\text{SNR}_{\text{input}} = 20 \log_{10} \left(\frac{\|L 
	\bar{x}\|}{\|L \bar{x} - c \|}   \right).
	\end{equation}
	Problem \eqref{eq:problem0o} is solved 
	using an  orthonormal Symmlet basis with  4 vanishing moments, and 2 
	resolution levels for $W$ operator. The following penalty parameter 
	values are chosen: $\lambda =150$, $\delta = 5$, and $\alpha =0.1$. 
	The reconstructed images using MMFBHF and MMFDRF with $10^4$ 
	iterations are presented in Figure~\ref{fig:reconstructed_wav}. We 
	also present the results within a zoomed region-of-interest (ROI), 
	with size $80 \times 80$ pixels and circular shape, in Figure 
	\ref{fig:reconstructed_wav} (bottom). We evaluate, for each 
	algorithm, the quantitative error between the original image 
	$\bar{x}$ and its recovered version $\hat{x}$, through the normalized 
	mean squared error NMSE $= {\|\bar{x} - \hat{x} 
\|^2}/{\|\bar{x}\|^2}$, the mean absolute error MAE $= \| \bar{x} - 
\hat{x}\|_{\infty}$ and the SNR $=  10 \log_{10} \text{NMSE}$. 
Similar formula are used to determine SNR, MAE and NMSE scores inside 
the ROI. The obtained values are provided in the caption of Figure 
\ref{fig:reconstructed_wav}.

In Figure~\ref{fig:graphcomparison_wav}, we display the evolution of 
the SNR, along iterations and times, for codes running in MATLAB 
R2023a, on a laptop with AMD Ryzen 5 3550Hz, Radeon Vega Mobile Gfx, 
and 32 Gb RAM. One can notice that  MMFBHF and MMFDRF behave 
similarly in terms of convergence 
speed. Still, all algorithms reach convergence in about $2000$ 
iterations, and $200$ seconds, confirming the validity of our 
theoretical results.

\begin{figure}
\centering
\subfloat[\scriptsize Original image 
$\overline{x}$]{\label{fig:original_wav}\includegraphics[scale=0.25]{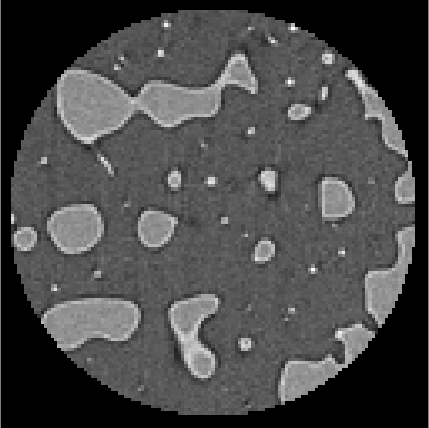}}\,
\subfloat[\scriptsize MMFBHF result, 
SNR=$23.29$ dB, NMSE = $4.71\times 10^{-3}$, MAE = $120.58 
$.]{\label{fig:FBHF_wav}\includegraphics[scale=0.25]{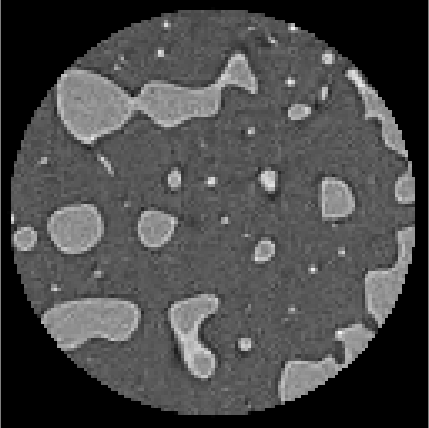}}\,
\subfloat[\scriptsize MMFDRF result, 
SNR=$23.20$ dB, NMSE = $4.72\times 10^{-3}$, MAE = 
$119.76$.]{\label{fig:FDRF_wav}\includegraphics[scale=0.25]{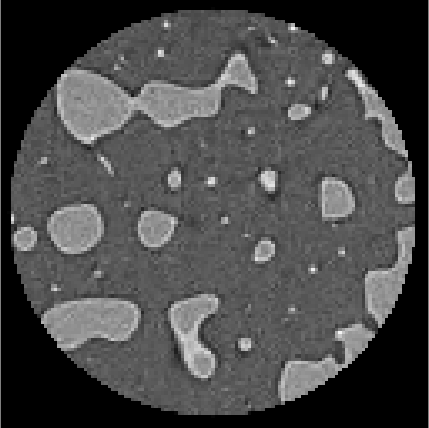}}\\
\subfloat[\scriptsize Zoomed original image 
$\overline{x}$]{\label{fig:originalzoom_wav}\includegraphics[scale=0.25]{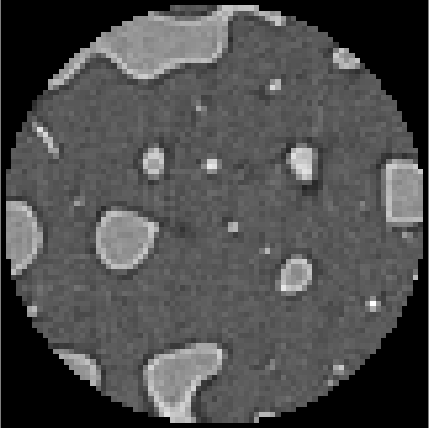}}\,
\subfloat[\scriptsize Zoomed MMFBHF result, 
SNR=$23.20$ dB, NMSE = $4.80\times 10^{-3}$, MAE = 
$120.58$.]{\label{fig:FBHF_zoom_wav}\includegraphics[scale=0.25]{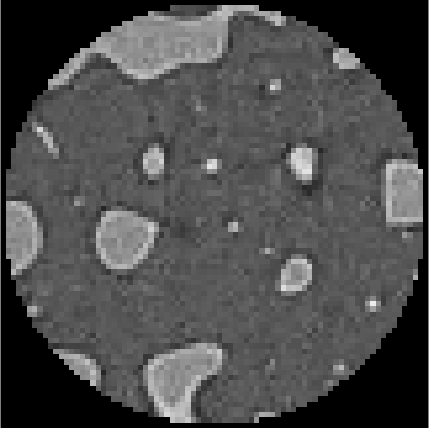}}\,
\subfloat[\scriptsize Zoomed MMFDRF result, 
SNR=$23.11$ dB, NMSE = $ 4.81\times 10^{-3}$, MAE = 
$116.90$.]{\label{fig:FDRF_zoom_wav}\includegraphics[scale=0.25]{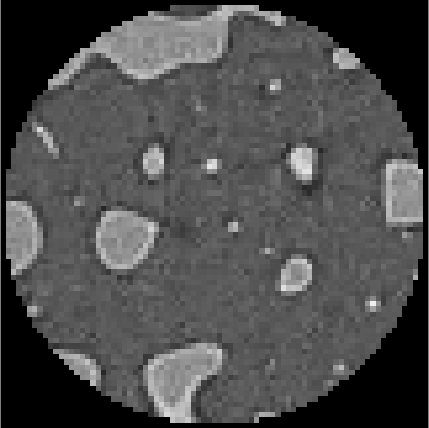}}
\caption{Original and reconstructed images (full view, and zoom) 
	after $10^5$ iterations of MMFBHF and MMFDRF algorithms, 
	respectively.}	\label{fig:reconstructed_wav}
	\end{figure}
	
	\begin{figure}
\centering
\begin{tabular}{@{}c@{}c@{}}
	\includegraphics[scale=0.43]{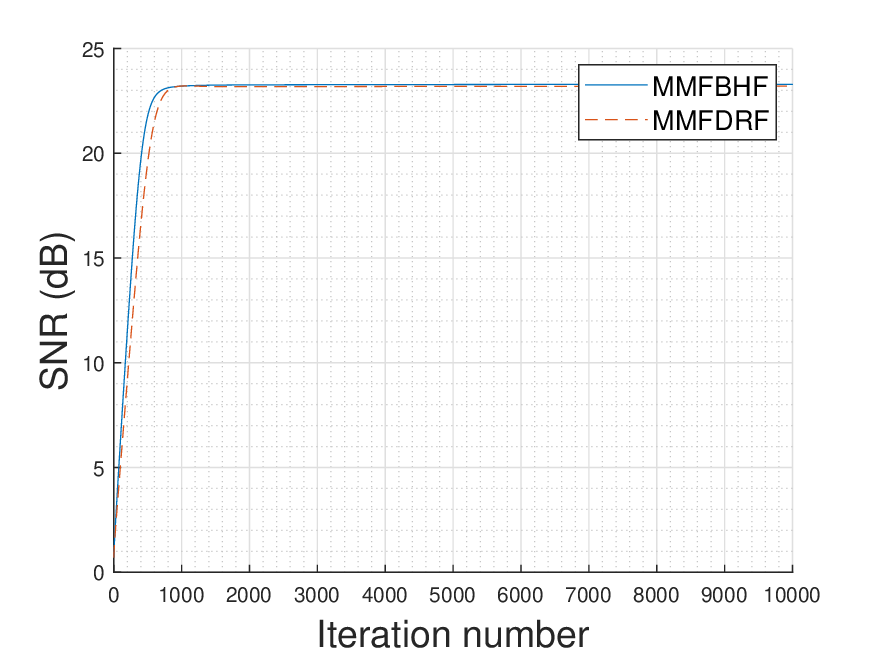} & 
	\includegraphics[scale=0.43]{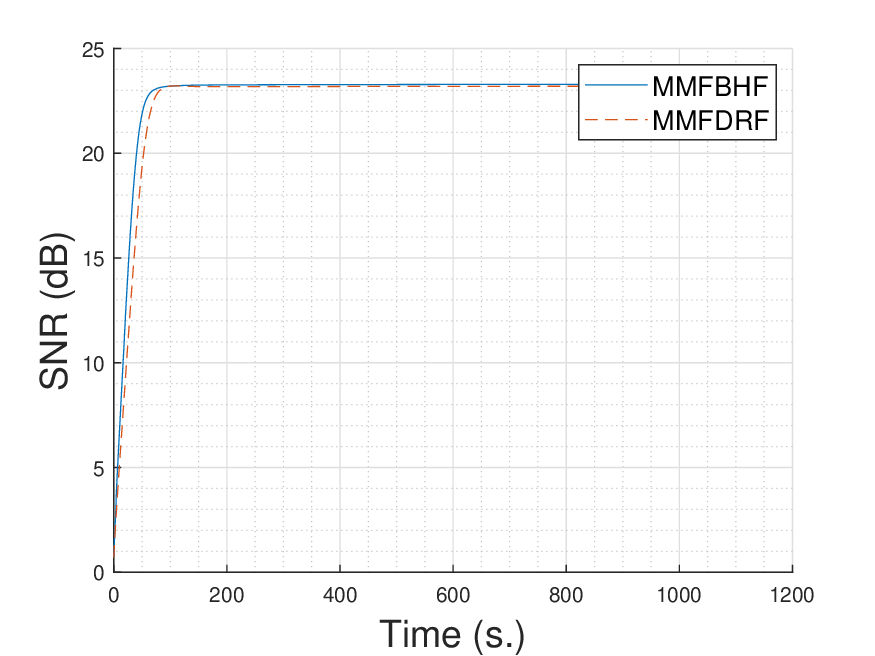}
\end{tabular} 
\caption{Evolution along iterations (left) and computational time 
	(right) in seconds, of the SNR (in dB) between the true image and its 
	reconstruction, for MMFBHF and MMFDRF 
	algorithms.}\label{fig:graphcomparison_wav}
	\end{figure}

\section{Conclusion}\label{se:conclu}
In this paper, we introduced two iterative algorithms for 
numerically solving monotone inclusions involving the sum of a 
maximally $\rho$-monotone operator, a cocoercive operator, and a 
mismatched Lipschitzian operator. The proposed schemes can be viewed 
as extensions of {\it Forward-Backward-Half-Forward} and {\it 
Forward-Douglas-Rachford-Forward} splitting methods, that use an 
approximation to an adjoint operator at each iteration. We provided 
conditions under which the sequence generated by these variants 
weakly converges to a solution to the mismatched inclusion. We also 
showed that, under some strong monotonicity assumptions, a linear 
convergence rate is obtained for the two algorithms. 
The applicability of our study is illustrated by numerical 
experiments in the context of imaging of materials.
When considering variational problems, the main advantage of
our work with respect to 
\cite{ChouzenouxMismatch2023,ChouzenouxMismatch2021,DongHansen2019,ElfvingHansen2018}
is to allow dealing with mismatches on more sophisticated functions 
than quadratic ones.

\section*{Acknowledgements}
E.C. acknowledges support from the European Research Council Starting 
Grant MAJORIS ERC-2019-STG-850925. The work by
J.-C.P. was supported by the ANR Research and Teaching Chair BRIDGEABLE in AI.

\end{document}